\documentclass[a4paper]{amsart}

\usepackage{amssymb,amscd}
\usepackage{amsmath,amsfonts,latexsym}
\usepackage{mathtools}
\usepackage[dvips]{graphicx}
\usepackage{color}
\usepackage[english]{babel}
\usepackage[latin1]{inputenc}
\usepackage{enumerate,afterpage}
\usepackage{tikz}
\usepackage{url}
\usetikzlibrary{matrix,arrows}

\usepackage{microtype}

\usepackage[all]{xy}

\def\co{\colon\thinspace} 

\newcommand{\Z}{\mathbb{Z}}
\newcommand{\CC}{\mathbb C}

\newcommand{\PO}{\operatorname{PO}}

\newcommand{\PGL}{\operatorname{PGL}}

\newcommand{\Lag}{{\mathrm{Lag}}}

\newcommand{\RR}{\mathbb R}

\newcommand{\Hom}{\operatorname{Hom}}

\newcommand{\Isom}{\operatorname{Isom}}

\newcommand{\CP}{\mathbb{CP}}
\newcommand{\KP}{\mathbb{KP}}

\def\co{\colon\thinspace}

\newcommand{\C}{{\mathbb C}}
\newcommand{\HH}{{\mathbb H}}
\newcommand{\K}{\mathbb{K}}
\newcommand{\N}{{\mathbb N}}
\newcommand{\R}{{\mathbb R}}
\newcommand{\bS}{{\mathbb S}}
\newcommand{\Sp}{{\mathrm{Sp}}}

\newcommand{\SL}{\mathrm{SL}}
\newcommand{\PSL}{\mathrm{PSL}}

\newcommand{\SO}{\mathrm{SO}}

\newcommand{\PU}{\mathrm{PU}}
\newcommand{\PSp}{\mathrm{PSp}}
\newcommand{\OG}{\mathrm{O}}

\def\co{\colon\thinspace}

\DeclareMathOperator{\Stab}{Stab}

\newcounter{notes}

\newcommand{\PP}{\mathbb{P}}
\newcommand{\PSU}{\mathrm{PSU}}
\newcommand{\ZZ}{\mathbb{Z}}

\newtheorem{Theorem}{Theorem}[section]
\newtheorem{Lemma}[Theorem]{Lemma}
\newtheorem{Proposition}[Theorem]{Proposition}
\newtheorem{Corollary}[Theorem]{Corollary}
\newtheorem{introthm}{Theorem}
\newtheorem{introlem}[introthm]{Lemma}
\newtheorem{introcor}[introthm]{Corollary}

\newtheorem{Definition}[Theorem]{Definition}

\newtheorem{Claim}[Theorem]{Claim}

\theoremstyle{remark}
\newtheorem{Remark}[Theorem]{Remark}
\newtheorem{Example}[Theorem]{Example}

\let\oldtocsection=\tocsection

\let\oldtocsubsection=\tocsubsection

\let\oldtocsubsubsection=\tocsubsubsection

\renewcommand{\tocsection}[2]{\hspace{0em}\oldtocsection{#1}{#2}}
\renewcommand{\tocsubsection}[2]{\hspace{1em}\oldtocsubsection{#1}{#2}}
\renewcommand{\tocsubsubsection}[2]{\hspace{2em}\oldtocsubsubsection{#1}{#2}}

\begin{document}


\title{Fiber bundles associated with Anosov representations}

\author[D. Alessandrini]{Daniele Alessandrini}
\address{Department of Mathematics, Columbia University}
\email{ daniele.alessandrini@gmail.com}
\urladdr{https://math.columbia.edu/$\sim$alessandrini}

\author[S. Maloni]{Sara Maloni}
\address{Department of Mathematics, University of Virginia}
\email{sm4cw@virginia.edu}
\urladdr{sites.google.com/view/sara-maloni/}

\author[N. Tholozan]{Nicolas Tholozan}
\address{D\'epartement de Math\'ematiques et Applications, \'{E}cole Normale Sup\'erieure -- PSL}
\email{nicolas.tholozan@ens.fr} 
\urladdr{https://www.math.ens.fr/$\sim$tholozan/}

\author[A. Wienhard]{Anna Wienhard}
\address{Max Planck Institute for Mathematics in the Sciences, Inselstr. 22, 04103 Leipzig, 
Germany}
\email{anna.wienhard@mis.mpg.de}
\urladdr{https://www.mathi.uni-heidelberg.de/$\sim$wienhard}

\thanks{We acknowledge support from U.S. National Science Foundation grants DMS 1107452, 1107263, 1107367 ``RNMS: GEometric structures And Representation varieties'' (the GEAR Network).  SM was partially supported by U.S. National Science Foundation grants DMS-1848346 (NSF CAREER). AW is supported by the European Research Council under ERC-Advanced Grant 101018839 and by the Deutsche Forschungsgemeinschaft (DFG, German Research Foundation) - Project-ID 281071066 - TRR 191. She thanks the Institute for Advanced Study for its hospitality when this work was completed.} 

\date{\today (v0)}

\begin{abstract}
Anosov representations of hyperbolic groups form a rich class of representations that are closely related to geometric structures on closed manifolds. Any Anosov representation $\rho:\Gamma \to G$ admits cocompact domains of discontinuity in flag varieties $G/Q$ \cite{guichard-wienhard:anosov, KLP2} endowing the compact quotient manifolds $M_\rho$ with a $(G,G/Q)$-structure. In general the topology of $M_\rho$ can be quite complicated. 

In this article, we will focus on the special case when $\Gamma$ is a the fundamental group of a closed (complex) hyperbolic manifold $N$ and $\rho$ is a deformation of a (twisted) lattice embedding $\Gamma \to \mathrm{Isom}^\circ (\mathbb{H}_{\mathbb{K}}) \to G$ through Anosov representations. In this case, we prove that $M_\rho$ is a smooth fiber bundle over $N$, we describe the structure group of this bundle and compute its invariants. This theorem applies in particular to most representations in higher rank Teichm\"uller spaces,  as well as convex divisible representations,  AdS-quasi-Fuchsian representations and $\mathbb{H}_{p,q}$-convex cocompact representations. 

Even when $M_\rho \to N$  is a fiber bundle, it is often very difficult to determine the fiber. In the second part of the paper  we focus on the special case when $N$ is a surface, $\rho$ a quasi-Hitchin representation into $\mathrm{Sp}(4,\C)$, and $M_\rho$ carries a $(\mathrm{Sp}(4,\C),\mathrm{Lag}(\C^4))$-structure. We show that in this case the fiber is homeomorphic to  $\C\mathbb{P}^2 \# \overline{\C\mathbb{P}}^2$. 

This fiber bundle $M_\rho \to N$ is of particular interest in the context of possible generalizations of Bers' double uniformization theorem in the context of higher rank Teichm\"uller spaces,  since for Hitchin-representations it contains two copies of the locally symmetric space associated to $\rho(\Gamma)$. Our result uses the classification of smooth $4$--manifolds, the study of the $\SL(2, \C)$--orbits of $\mathrm{Lag}(\C^4)$ and the identification of $\mathrm{Lag}(\C^4)$ with the space of (unlabelled) regular ideal hyperbolic tetrahedra and their degenerations. 
\end{abstract}

\maketitle

\tableofcontents

\section{Introduction} 
A $(G,X)$-manifold is a topological manifold that is locally modelled on a $G$-homogeneous space $X$. This means that the manifold is equipped with local charts with values in a model space $X$ and transition functions with values in a Lie group $G$ acting transitively on $X$. $(G,X)$-manifolds play an important role in Thurston's geometrization program, but extend far beyond it. In particular $(G,X)$-structures capture many geometric structures beyond Riemannian metrics, for example projective or affine structures on manifolds.
 
The simplest examples of $(G,X)$-manifolds are quotients of $X$ by a discrete subgroup of $G$ acting freely and properly discontinuously -- the \emph{complete} $(G,X)$-manifolds. A bigger class of examples arise more generally from quotients of an open domain $\Omega$ of $X$. Theses are sometimes called \emph{Kleinian} $(G,X)$-manifolds.

%

The terminology \emph{Kleinian} comes from the theory of \emph{Kleinian groups}. A Kleinian group $\Gamma$ is a (non elementary) discrete subgroup of isometries of the hyperbolic $3$-space $\HH^3$. Its action on the boundary at infinity $\partial_\infty \HH^3 \simeq \mathbb{CP}^1$ has a minimal invariant \emph{limit set} $\Lambda_\Gamma$ and is properly discontinuous on the complement $\Omega_\Gamma$, called the domain of discontinuity. The quotient $\Gamma \backslash \Omega_\Gamma$ is (at least when $\Gamma$ is torsion-free) a Riemann surface equipped with a \emph{complex projective structure}.
 The group $\Gamma$ is \emph{convex-cocompact} when the action of $\Gamma$ on $\HH^3 \sqcup \Omega_\Gamma$ is cocompact. The quotient then gives a conformal compactification of the hyperbolic $3$-manifold $\Gamma \backslash \HH^3$. Convex-cocompact Kleinian groups and the corresponding hyperbolic $3$-manifolds play a central role in Thurston's hyperbolization theorem.
 
In recent years, the theory of convex-cocompact subgroups in rank one Lie groups has been generalized to the higher rank setting through the theory of Anosov representations, introduced by Labourie \cite{labourie-anosov} and, more generally, by Guichard-Wienhard \cite{guichard-wienhard:anosov}. A strong connection between Ansosov representations and $(G,X)$-manifolds has been established through the construction of domains of discontinuity by Guichard--Wienhard \cite{guichard-wienhard:anosov} and then by Kapovich--Leeb--Porti \cite{KLP2}. 
%

Let $G$ be a semisimple Lie group and $P$ a parabolic subgroup of $G$. Informally, a \emph{$P$-Anosov representation} of a Gromov hyperbolic group $\Gamma$ into $G$ is a homomorphism $\rho:\Gamma \to G$ that admits a $\rho$-equivariant continuous embedding of the boundary at infinity $\partial_\infty \Gamma$ of $\Gamma$ into the flag variety $G/P$, which preserves the dynamics of the action of $\Gamma$ on its boundary (see Definition \ref{def - Anosov Representation}). $P$-Anosov representations are quasi-isometric embeddings and are stable under small perturbations. Guichard--Wienhard proved in \cite{guichard-wienhard:anosov} that, for some parabolic subgroup $Q$ (possibly different from $P$), a $P$-Anosov representation $\rho: \Gamma \to G$ defines a properly discontinuous and cocompact action of $\Gamma$ on an open subset $\Omega_\rho \subset G/Q$, which is the complement of a disjoint union of Schubert subvarieties parametrized by $\partial_\infty \Gamma$ (see Definition \ref{thm - Domains of discontinuity}). The precise parabolic subgroups $Q$ for which this construction works have been systematically described by Kapovich--Leeb--Porti \cite{KLP2}. We call the domains of discontinuity $\Omega_\rho$ obtained by their constructions \emph{flag domains of discontinuity}.

This associates to a torsion-free Anosov representation $\rho$ a closed manifold $M_\rho= \Gamma \backslash \Omega_\rho$ equipped with a Kleinian $(G,G/Q)$-structure. Even though the topological type of $M_\rho$ only depends on the connected component of $\rho$ in the space of $P$-Anosov representations \cite{guichard-wienhard:anosov}, it is very difficult to determine the topology of $M_\rho$. 
%
Even in the case of convex-cocompact subgroups of a hyperbolic space of dimension $\geq 4$, the topology of $M_\rho$  is still very mysterious, and many wild phenomena can occur (see Section \ref{subsec:GromovLawsonThurston} for an example with a surface group), suggesting that a systematic answer is impossible. Nonetheless, important classes of examples of Anosov representations arise from deformations of uniform lattices in Lie groups of real rank $1$ into a higher rank Lie group. For such representations, the topology of $M_\rho$ becomes more tractable.

\subsection*{Part 1: Deformations of rank 1 lattices}

In the first part of the paper, we prove a general fibration theorem for the quotients of flag domains of discontinuity associated to Anosov deformations of a rank $1$ lattice into a higher rank Lie group.

Let $H$ be a connected semisimple Lie group of real rank $1$ with finite center, and $\Gamma$ a torsion-free uniform lattice in $H$. Denote by $\rho_0\co\Gamma \to H$ the identity representation. Let $G$ be a connected semisimple Lie group with finite center, and $\iota\co H \to G$ a representation. Then $\iota \circ \rho_0$ is a $P$-Anosov representation of $\Gamma$ in $G$ for some parabolic subgroup $P$: we will call $\iota \circ\rho_0$ an $\iota$-lattice representation. The set $\mathrm{Anosov}_{P}(\Gamma,G)$ of $P$-Anosov representations of $\Gamma$ into $G$ is an open subset of $\Hom(\Gamma, G)$ containing the $\iota$-lattice representation $\iota\circ\rho_0$. We call a representation $\rho:\Gamma \to G$ a \emph{$P$-Anosov deformation of $\iota\circ\rho_0$} if $\rho$ belongs to the connected component of $\iota\circ\rho_0$ in $\mathrm{Anosov}_P(\Gamma,G)$.

Let us now fix any parabolic subgroup $Q$ of $G$ such that $P$-Anosov representations of $G$ admit a cocompact flag domain of discontinuity in $G/Q$. For $\rho \in \Hom_{P}(\Gamma,G)$, we denote the domain by $\Omega_\rho \subset G/Q$  and the closed quotient manifold by $M_\rho = \rho(\Gamma) \backslash \Omega_\rho$. Finally, we denote by $S_H$ the symmetric space of $H$. Our first main result is the following:

\begin{introthm}[see Theorem \ref{thm:TopologyQuotient}]\label{thm_intro_main}
For $H,\Gamma,G,\iota,P,Q$ as above, let $\rho:\Gamma \rightarrow G$ be a $P$-Anosov deformation of $\iota\circ\rho_0$. Then $\Omega_\rho$ admits a smooth $\Gamma$-equivariant fibration onto $S_H$. In particular, $M_\rho$ is a smooth fiber bundle over the negatively curved locally symmetric space $\Gamma \backslash S_H$, and $\Omega_\rho$ deformation retracts to a closed manifold of dimension $\dim(G/Q)-\dim(S_H)$.
\end{introthm}

This theorem comes with companion theorems (see Theorem~\ref{thm:BundleStructureQuotient} and Corollary~\ref{cor:TopologyBundleSurfaces} below) that describe the structure group and the invariants of the fiber bundle. 

Let us emphasize that Theorem~\ref{thm_intro_main} applies to every (!) cocompact flag domain of discontinuity. Even for a given representation, and a given flag variety, there can be many different cocompact flag domains of discontinuity (see \cite{Stecker}). 
The theorem further applies when we consider the representation $\rho:\Gamma \rightarrow G \rightarrow G'$ as an Ansosov representation into a larger Lie group, and all the flag domains of discontinuity constructed in flag varieties of $G'$. See \cite{guichard-wienhard:anosov} and \cite{GGKW} for examples of how one can play around with such embeddings into larger groups.

An important source of applications of Theorem~\ref{thm_intro_main} is when an entire connected component of the representation variety $\Hom(\Gamma, G)$ consists of Anosov representations. For fundamental groups of closed surfaces with negative Euler characteristic (which we call \emph{surface groups} from now on), such components are called higher rank Teichm\"uller components, and most of them contain representations that factor through a compact extension of $\PSL(2,\RR)$. Therefore, Theorem~\ref{thm_intro_main} gives a positive answer to \cite[Conjecture~13]{Wienhard_ICM} in most cases, as well as a complete answer to a conjecture by Dumas--Sanders \cite[Conjecture~1.1]{Dumas_Sanders_1}. 
More precisely we have 
\begin{introcor}\label{cor_intro}
Let $\Gamma = \pi_1(\Sigma)$ be a surface group, and $\mathcal{C}\subset \Hom(\Gamma, G)$ be a higher rank Teichm\"uller component that contains a twisted Fuchsian representation. Then for every parabolic subgroup $Q$ and every cocompact flag domain of discontinuity $\Omega \subset G/Q$ the quotient manifold $M = \pi_1(\Sigma) \backslash\Omega$ is homeomorphic to a fiber bundle $M\to \Sigma$. In addition, for all $\rho \in \mathcal C$, the composition $\pi_1(M) \to \pi_1(\Sigma) \overset{\rho}{\to} G$ is the holonomy of a Kleinian $(G,G/Q)$-structure on $M$.
\end{introcor}

We explain Corollary~\ref{cor_intro} and further examples in a bit more detail.
\begin{enumerate}
\item {\bf Hitchin components.}
Let $\Gamma = \pi_1(\Sigma)$ be a surface group and $G$ a split real simple linear group. The group $H = \SL(2,\RR)$ admits a \emph{principal representation} $\iota_0:H \rightarrow G$. Given a Fuchsian representation $\rho_0:\Gamma \rightarrow H$, the composition $\iota_0\circ\rho_0$ is called a principal Fuchsian representation in $G$. The representations of the connected component of $\iota_0\circ\rho_0$ in $\mathrm{Hom}(\Gamma,G)$ are called \emph{Hitchin representations}. Hitchin representations are Anosov with respect to the minimal parabolic subgroup $P_{min}<G$, or, equivalently, with respect to any parabolic subgroup $P<G$ \cite{labourie-anosov, fock-goncharov-1}. They are thus $P$-Anosov deformations of the lattice representation $\iota_0\circ\rho_0$, and Theorem~\ref{thm_intro_main} applies.

\item{\bf $P$-quasi-Hitchin representations.} 
Theorem~\ref{thm_intro_main}  also applies to deformations of Hitchin representations into complex Lie groups. For this we embed $G$ into its complexification $G_{\CC}$ and we consider the principal Fuchsian representation  $\iota_0\circ\rho_0:\Gamma \to G <G_{\CC}$ as taking values in $G_{\CC}$. This representation is Anosov with respect to any parabolic subgroup $P_{\CC} < G_{\CC}$. However not every continuous deformation of $\iota_0\circ\rho_0$ will be Anosov. We define the set of $P$-quasi-Hitchin representation to the be connected component of the space of 
 $P$-Anosov representations in $G_{\CC}$ containing the principal Fuchsian representations. When $G = \PSL(2,\RR)$ this is precisely the set of quasi-Fuchsian representations. Note that, in higher rank, this set might depend on the choice of parabolic subgroup $P$.
 
The geometry of $P_{min}$-quasi-Hitchin representations in $\PSL(n,\C)$ has been studied by Dumas--Sanders \cite{Dumas_Sanders_1}. In particular, they proved that flag domains of discontinuity $\Omega_\rho$ satisfy a Poincar\'e duality of rank $\dim(G/Q)-2$ and computed the cohomology of $M_\rho$. They conjectured that $M_\rho$ admits a fibration over the surface $\Sigma$.  Theorem~\ref{thm_intro_main} applies in this situation and thus gives a positive answer to their conjecture.

%

\item{\bf Positive representations.}
The Hitchin component is one example of a higher rank Teichm\"uller component. Other examples are formed by maximal representations, and more generally by spaces of $\Theta$-positive representations introduced in \cite{GW_pos, GW_pos2}. Here $\Theta$ is a subset of the set of simple roots $\Delta$. Hitchin representations are $\Delta$-positive representations. Maximal representations into Hermitian groups of tube type are $\{\alpha\}$-positive for a specific choice of $\alpha \in \Delta$. There are two further families of Lie groups that admit $\Theta$-positive structures and $\Theta$-positive representations.
When $G$ is a Lie group carrying a $\Theta$-positive structure, then there is a a special simple three dimensional $\Theta$-principal subgroup in $G$, see \cite{GW_pos2}. Contrary to the principal subgroup of a split real Lie group, this $H$ might have a compact centralizer, so there is a compact extension $H$ of the $\Theta$-principal subgroup that embeds into $G$.  We choose a discrete and faithful representation $\rho_0:\pi_1(\Sigma) \rightarrow H$ and call $\iota\circ\rho_0:\pi_1(\Sigma) \to G$ a twisted $\Theta$-principal embedding. This representation is $P_\Theta$-Anosov, where $P_\Theta$ is the parabolic subgroup determined by $\Theta$. In fact any deformation of $\iota\circ\rho_0$ is $P_\Theta$-Anosov \cite{GLW_pos}, and thus Theorem~\ref{thm_intro_main} applies.

Note that there are cases where not every $\Theta$-positive representation arises from a deformation of a principal or $\Theta$-principal Fuchsian representation. In particular when $G = \Sp(4,\RR), \SO(2,3), \SO(n,n+1)$ there are connected components of the space of $\Theta$-positive representations where every representation is Zariski-dense. In particular, Theorem~\ref{thm_intro_main}  does not apply to these components. When $G = \Sp(4,\RR), \SO(2,3)$ it has been proven by other means that the quotient manifolds $M_\rho$ are fiber bundles over $\Sigma$, see the discussion in Section~\ref{sec:related}.

\item{\bf $P$-quasi-positive representations.} 
Similar as in the discussion of quasi-Hitchin representations, when $G$ admits a $\Theta$-positive structure, we can embed $G$ into its complexification $G_\CC$, and any $\Theta$-positive representation $\rho:\pi_1(\Sigma) \to G <G_\CC$ will be $P$-Anosov for a set of parabolic subgroups determined by $\Theta$. Thus we can define the set of $P$-quasi-positive representations as the connected components of the space of $P$-Anosov representations into $G_\CC$ containing a $\Theta$-positive representation into $G$. 
Theorem~\ref{thm_intro_main} then applies to the components of  $P$-quasi-positive representations that contain a twisted $\Theta$-principal embedding.

\item{\bf Convex divisible representations.}
Applications of Theorem~\ref{thm_intro_main} are not limited to representations of surface groups, there are also interesting classes of representations of fundamental groups $\pi_1(M)$ of closed hyperbolic manifolds $M$ of higher dimension. These groups come together with a representation $\rho_0:\pi_1(M) \to \PO(1,n)$.  
Benoist \cite{Benoist1} introduced the notion of convex divisible representations. These are representations $\Gamma \to \PGL(n+1,\RR)$ for which there exists a $\Gamma$-invariant strictly convex domain $\Omega \subset \RR\PP^{n}$ on which $\Gamma$ acts properly discontinuously and cocompactly. Examples of such representations arise from the embedding $\iota:\PO(1,n) < \PGL(n+1,\RR)$, by considering the lattice representation $\iota\circ\rho_0$. This representation is $P_{1,n}$-Anosov, where $P_{1,n}$ is the stabilizer of a partial flag consisting of a line contained in a hyperplane. Benoist \cite{Benoist3} showed that the set of convex divisible representations is open and closed in the representation variety. In particular, any deformation of $\iota\circ\rho_0$ is a  $P_{1,n}$-Anosov representation, and there are many hyperbolic $n$-manifolds that admit such deformations. 
The topology of the domain 
$\Omega \subset \RR\PP^{n}$ and its quotient $\pi_1(M) \backslash \Omega$ is easy to understand. $\Omega$ is contractible and so $\pi_1(M) \backslash \Omega \cong M$. However, $\Omega$ is not the only flag domain of discontinuity one can associate to a convex divisible representation. There are many flag domains of discontinuity in other flag varieties $G/Q$, as well as flag domains of discontinuities in $G'/Q' $ when we consider $\pi_1(M)$ in a bigger Lie groups $G'$. The topology of these domains of discontinuity can be more complicated. Theorem~\ref{thm_intro_main} applies to all these domains of discontinuity, as long as $\rho$ is in the same connected component (of Anosov representations) as $\iota\circ\rho_0$.

\item{\bf AdS-quasi Fuchsian representations.} 
In \cite{Barbot-15,barbot-merigot} Barbot introduced the notion of AdS-quasi-Fuchsian representations of fundamental groups $\pi_1(M)$ of closed hyperbolic manifolds $M$ into $\SO(2,n)$. For this he considers the embedding $\iota:\SO(1,n) < \SO(2,n)$. The representation $\iota\circ\rho_0$ is $P$-Anosov, where $P$ is the stabilizer of an isotropic line. Barbot and Merigot showed that any deformation of $\iota\circ\rho_0$ is a $P$-Anosov representation, and moreover any AdS-quasi-Fuchsian representations is a deformation of $\iota\circ\rho_0$. Thus Theorem~\ref{thm_intro_main} applies to all AdS-quasi-Fuchsian representations.

\item{\bf Convex-cocompact $\mathbb{H}^{p,q}$-representations.}
Barbot's construction can be generalized. When $\pi_1(M)$ is the fundamental group of a closed hyperbolic manifold $M$, we can consider the embedding 
$\iota: \SO(1,q)< \SO(p,q)$. The representation $\iota\circ\rho_0$ is $P$-Anosov, where $P$ is the stabilizer of an isotropic line. It gives rise to a $\mathbb{H}^{p.q}$-convex-cocompact representation, studied by \cite{DGK18}. It is expected (see \cite{Wienhard_ICM}) that the entire connected component containing $\iota\circ\rho_0$ consists of 
$P$-Anosov representations, if this is proved, then Theorem~\ref{thm_intro_main} would apply to all representations in the component of $\iota\circ\rho_0$.

Note finally that one can consider Anosov deformations of these representations into the complexified group, and Theorem~\ref{thm_intro_main} applies in to such complex deformations as well.
\end{enumerate}

Even if we know that $M_\rho$ is a fiber bundle over  the locally symmetric space $\Gamma\backslash S_H$, it seems difficult in general to determine precisely the topology of the fiber. Explicit descriptions of the fibers have been given in some cases, see Section~\ref{sec:related}. In fact the main reason why such a general result as Theorem~\ref{thm_intro_main} has been previously overlooked seems to be that, in interesting low dimensional situations, there are explicit and natural $H$-equivariant fibrations from $\Omega$ to $S_H$ which are not smooth and whose fibers are not manifolds.

In the proof of Theorem~\ref{thm_intro_main}, the assumption that $\rho$ is a $P$-Anosov deformation of a rank one lattice is used crucially in order to reduce to the ``Fuchsian'' case. Indeed, Guichard--Wienhard proved that the topology of  $M_\rho$ is invariant under continuous deformation of $\rho$ in $\mathrm{Anosov}_P(\Gamma,G)$. We can thus assume without loss of generality that $\rho = \iota\circ\rho_0$. In that case, the domain $\Omega_\rho$ is $H$-invariant, and our main theorem follows from the following general result:

\begin{introlem}[See Lemma \ref{lem:proper action}]
Let $X$ be a smooth manifold with a proper action of a semisimple Lie group $H$. Then there exists a smooth $H$-equivariant fibration from $X$ to the symmetric space $S_H$. 
\end{introlem}

Though this fairly general lemma sounds like a classical result, it seems to have been overlooked by people in the field. To prove it, we fix an arbitrary torsion-free uniform lattice $\Gamma \subset H$, choose a smooth $\Gamma$-equivariant map from $X$ to $S_H$, and then take a barycentric average of $f$ under some action of $H$.

A more precise version of Theorem~\ref{thm_intro_main} (see Theorem~\ref{thm:BundleStructureQuotient} and Corollary~\ref{cor:TopologyBundleSurfaces}) shows that $M_\rho$ is a fiber bundle over $\Gamma \backslash S_H$ associated to an explicit principal $K$-bundle, where $K$ is a maximal compact subgroup of $H$. In order to complete the description of $M_\rho$, the only missing element is the topology of the fiber. The topology of the fiber has been determined in some cases, see Section~\ref{sec:related}. 
In the second part of the paper we determine the fiber in a special low-dimensional case.

\subsection*{Part 2: Symplectic quasi-Hitchin representations}

In the second part of the paper, we focus on  \emph{$P$-quasi-Hitchin representations} into $\PSp(4,\CC)$, where $P$ is the stabilizer of a line in $\mathbb{CP}^3$. Let $\Gamma$ be the fundamental group of a closed surface $\Sigma$ of genus $g\geq 2$, embedded as a uniform lattice in $H= \PSL(2,\R)$ via a Fuchsian representation $\rho_0$. Let $\iota_0\co\PSL(2,\R)\to \PSp(4,\CC)$ the principal  representation. We see $\iota_0\circ\rho_0$ as a $P$-Anosov representation, and we consider $P$-quasi-Hitchin representations, i.e. $P$-Anosov deformations of  $\iota_0\circ\rho_0$. 

Guichard and Wienhard \cite{guichard-wienhard:anosov} show that $P$-quasi-Hitchin representations $\rho\co\pi_1(\Sigma)\to\mathrm{Sp}(4,\CC)$ admit cocompact domains of discontinuity $\Omega_\rho$ in the space $\Lag(\CC^{4})$ of Lagrangian subspaces of $\CC^{4}$, of complex dimension $3$. We write as before $M_\rho = \rho(\Gamma) \backslash \Omega_\rho$. By topological invariance $M_\rho$ is diffeomorphic to $M_{\iota_0\circ\rho_0}$, and Theorem~\ref{thm_intro_main} tells us that this manifold is a smooth fiber bundle over the hyperbolic surface $\Sigma = \Gamma \backslash \HH^2$.

We prove the following theorem:
\begin{introthm}\label{fiber_smooth}
Let $\rho$ be a $P$-quasi-Hitchin representation of a surface group $\Gamma= \pi_1(\Sigma)$ into $\PSp(4,\CC)$, and let $\Omega_\rho$ be its flag domain of discontinuity in the space of complex Lagrangians. Then $M_\rho = \rho(\Gamma) \backslash \Omega_\rho$ is a smooth fiber bundle over $\Sigma$ with fiber homeomorphic to $\CP^2 \# \overline{\CP}^2$.
\end{introthm} 

The domain of discontinuity $\Omega_\rho \subset \Lag(\CC^{2n})$ is of particular interest in the context of potential generalizations of Bers' double uniformization for higher rank Teichm\"uller spaces.
In the case when $n=1$ and $\rho$ is a Fuchsian representation, $\Omega_\rho$ is the disjoint union of the upper and the lower half disc, if $\rho$ is a quasi-Fuchsian representation it is precisely the complement of the limit set, and thus consists of two connected components, whose quotients give rise to the two conformal structures associated to a quasi-Fuchsian representation.
For general $n$ and $\rho$ a Hitchin representation into $\PSp(2n,\RR)$, the domain of discontinuity $\Omega_\rho$ contains two copies of the symmetric space associated to $\PSp(2n,\RR)$, a copy of the Siegel upper half space, and a copy of the Siegel lower half space, which are exchanged by the complex conjugation. On the other hand it also contains other strata, e.g. all the pseudo-Riemannian symmetric spaces $\PSp(2n,\RR)/\PSU(p,q)$, $p+q=n$, which are permuted by the complex conjugation. (For a more detailed discussion see \cite{Wienhard_thurston}). 
The fact that for $\PSp(4,\CC)$ the fiber is $\CP^2 \# \overline{\CP}^2$ appears to be quite interesting in this respect.

In order to prove Theorem~\ref{fiber_smooth} we actually have to take quite a bit of a detour. We first give a natural geometric construction of an $H$-equivariant continuous fibration $\pi$ from $\Omega_{\iota_0\circ\rho_0}$ to $\HH^2$. The map $\pi$ is not smooth and its fiber $F$ is singular. Nevertheless, the fiber $F$ is homotopy equivalent to the fiber $F'$ of a smooth equivariant fibration, since both are retractions of $\Omega_{\iota_0\circ\rho_0}$. By carefully studying $F$, we can determine its second cohomology and the intersection form on it. Finally, using the classification of smooth $4$-manifolds due to Whitehead, Milnor, Milnor--Hausemoller, Freedman, Serre and Donaldson, we deduce the homeomorphism type of $F'$ (which has the same second homology group) and prove the theorem.

\subsection{Related works and perspectives}\label{sec:related}


The topology of flag domains of discontinuity and their quotient manifolds $M_\rho$ for Anosov representations $\rho$ have been studied before in special examples, mainly for Anosov representations of a surface group $\pi_1(\Sigma)$. We review these results here.

In \cite{guichard-wienhard:convex} Guichard--Wienhard constructed flag domains of discontinuity in $\RR\PP^3$ for Hitchin representations into $\PSL(4,\RR)$ and $\PSp(4,\RR)$. They showed that these domains of discontinuity have two connected components $\Omega_1$ and $\Omega_2$. They showed that the quotient manifold $\pi_1(S) \backslash \Omega_1$ is homeomorphic to the unit tangent bundle $T^1S$ of the surface and in fact gives rise to convex foliated projective structures of $T^1S$. The quotient manifold $\pi_1(S) \backslash \Omega_2$ is a quotient of  $T^1S$ by $\ZZ/3\ZZ$.
They also show that deformations of quasi-Fuchsian representations (in $\PSL(2,\CC)\cong \PO(3,1)$) into $\PSL(4,\RR)$ give rise to projective structures on $T^1 S$. 

The study of Hitchin representations in  $\PSL(4,\RR)$ and $\PSp(4,\RR)$ and their domains of discontinuity in $\RR\PP^3$ can be carried out also for lattices in $\PSL(2,\R)$ that have torsion, see Alessandrini--Lee--Schaffhauser \cite{AleLeeSchaff}. There, they show that in this case the quotient $\RR\PP^3$-manifolds are homeomorphic to certain Seifert-fibered 3-manifolds that depend on the lattice.    

In \cite{guichard-wienhard:anosov}, determining part of the cohomology of the flag domains of discontinuity played a key role in showing that the action of $\rho(\Gamma)$ on $\Omega_\rho$ is cocompact. 
In their description of examples of such flag domains of discontinuity they describe several explicit examples, among them some where $M_\rho$ are in fact Clifford--Klein forms.
For maximal representations in the symplectic group, and for the domain of discontinuity in $\RR\PP^{2n-1}$ they announced that $M_\rho$ is a fiber bundle over $S$ with fiber $\OG(n)/\OG(n-2)$. This in particular also applies to the components of the space of maximal representations into $\PSp(4,\RR)$ where all represenations are Zariski-dense. This result lead them to conjecture that the quotient manifold $M_\rho$ is a compact fibre bundle over $\Sigma$ for all higher Teichm\"uller spaces, see \cite[Conjecture~13]{Wienhard_ICM}.


In \cite{CTT}, Collier--Tholozan--Toulisse studied the case where $\rho: \pi_1(\Sigma) \to \SO(2,n+1)$ is a maximal representation of a closed surface group. Such representations admit a flag domain of discontinuity $\Omega_\rho$ in the space of totally isotropic planes in $\R^{2,n+1}$. The authors prove that such maximal representations come with an equivariant spacelike embedding of $\HH^2$ into the pseudo-hyperbolic space $\HH^{2,n}$ and that the domain $\Omega_\rho$  fibers $\rho$-equivariantly over this spacelike disk, and deduce that $M_\rho$ is a homogeneous fiber bundle over $\Sigma$ with fiber a Stiefel manifold. The topological invariants of this fiber bundle turn out to depend on the connected component of $\rho$ is the set of maximal representations.

In particular, for $n=3$, one obtains circle bundles over $\Sigma$ whose Euler class varies with the connected component of maximal representations. Interestingly, there are connected components of the set of maximal representations into $\SO(2,3)$ that do not contain a representation factoring through $\PSL(2,\R)$. For these representations, The fibration of $M_\rho$ over $\Sigma$ is not given by Theorem \ref{thm_intro_main}.


%
%
%
%

When $\rho$ is a quasi-Hitchin representation into a complex group $G$, Dumas--Sanders \cite{Dumas_Sanders_1} computed the cohomology ring of $\Omega_\rho$ and $M_\rho$ for all choices of parabolic subgroups and balanced ideals. They found that the cohomology of $M_\rho$ is the tensor product of the cohomology of $\Sigma$ with the cohomology of $\Omega_\rho$ and that, under their hypothesis, $\Omega_\rho$ is a Poincar\'e duality space. They remarked that this is compatible with $M_\rho$ being a fiber bundle on $\Sigma$, and stated a conjecture \cite[Conjecture~1.1]{Dumas_Sanders_1} that is a special case of our Theorem~\ref{thm_intro_main}. Interestingly, in their conjecture they stated that $M_\rho$ is a continuous fiber bundle over $\Sigma$ because, in some examples available at the time, the known fibrations where only continuous, but not smooth. They verified their conjecture in the special case when $G=\SL(3,\C)$ and $G/Q$ is the full flag variety.


When $G = \SL(2n,\K)$ with $\K = \R$ or $\C$, $\iota$ is the principal representation and $G/Q$ is $\KP^{2n-1}$, Alessandrini--Davalo--Li \cite{AleDavLi} described the topology of $M$, as fiber bundle over $\Sigma$, and described the topology of the fiber, the structure group $\SO(2)$ and the Euler class. They used Higgs bundles, as described in the survey paper \cite{AleSIGMA}. In a paper in preparation, Alessandrini--Li \cite{AleLi} extend some of these results to the case when $G = \SL(n,\K)$ and $G/Q$ is a partial flag manifold parametrizing flags consisting of lines and hyperplanes, and when $G = \SL(4n+3,\R)$, $G/Q = \bS^{4n+2}$, and $M$ is the manifold constructed by Stecker-Treib \cite{ST}.

In an independent work using different techniques, Colin Davalo \cite{Davalo} proves related results. Given an $\iota$-Fuchsian representation of a surface group, under certain hypotheses, he selects a suitable parabolic subgroup $Q$ and he can describe a cocompact domain of discontinuity in $G/Q$. He proves that the quotient manifold of that domain is a fiber bundle over the surface. For example, for each positive $\iota$-Fuchsian representation, he can describe one or two such domains. He also shows two examples of $\iota$-lattice representations of fundamental groups of hyperbolic manifolds where his technique still works.


%
%
%
%

\subsubsection{Wild Kleinian groups} \label{subsec:GromovLawsonThurston}

In \cite{Gromov_Lawson_Thurston} Gromov--Lawson--Thurston show that one can obtain wild convex-cocompact embeddings of a surface group $\Gamma = \pi_1(\Sigma)$ into $\Isom(\HH^4)$ from a ``twisted necklace'' of $2$-spheres in $\partial_\infty \HH^4$. They construct such convex-cocompact representations for which $M_\rho$ is a non-trivial circle bundle over~$\Sigma$. Again, by topological invariance, such $\rho$ cannot be deformed to a Fuchsian representation within the domain of convex-cocompact representations. Such examples were also obtained independently by Kapovich \cite{kapovich:flat}. 

Gromov--Lawson--Thurston also point out that, starting from a knotted necklace, one obtains a convex-cocompact representation whose limit set is a \emph{wild knot}. The assiociated conformal $3$-manifold $M_\rho$ is then obtained by gluing a circle bundle over a surface with boundary with one or several knot complements. These examples do not fiber over the surface $\Sigma$ and their domain of discontinuity has infinitely generated fundamental group, showing that Theorem \ref{thm_intro_main} cannot be true in general for Anosov representations which are not Fuchsian deformations.

For more examples of convex-cocompact subgroups of $\Isom(\HH^n)$ with ``wild'' limit set (e.g. Antoine's necklace of Alexander's horned sphere), we refer to the survey of Kapovich \cite{kapovich:HigherKleinian}.

\subsection*{Acknowledgments}

The authors would also like to thank Renato Bettiol, David Dumas, Steve Kerckhoff, Qiongling Li, Tom Mark, Tye Lidman, Andy Sanders, Florian Stecker for interesting conversations related to the paper, and Colin Davalo for sharing a preprint of his results.

\subsection*{Outline of the paper}

In Section \ref{sec:background}, we review the required background on Anosov representations and their domains of discontinuity. Section \ref{sec:topology_of} is dedicated to the proof of Theorem \ref{thm_intro_main}. These form the first part of the paper.

The second part on the paper focuses on quasi-Hitchin representations in $\Sp(4,\CC)$. In Section \ref{sl2orbits}, we describe the action of $\PSL(2,\C)$ on the $\Lag(\C^4)$ and identify the Lagrangian Grassmannian to the space of (possibly degenerate) regular ideal tetrahedra in $\HH^3$. Using this point of view, we construct a $\PSL(2,\R)$-equivariant ``projection'' from $\Lag(\C^4)$ to $\bar{\HH^2}$ that we study more closely in Section \ref{sub:new_york}. In Section \ref{fiber} we carefully study the topology of the fiber $F$ of this projection. In particular, we compute the intersection form on its second cohomology group, and conclude the proof of Theorem \ref{fiber_smooth} using the topological classification of simply connected $4$-manifolds.


\part{Topology of the quotient of the domain of discontinuity } \label{part1}

\section{Anosov representations} \label{sec:background}


In this section, we recall the notion of Anosov representation, originally introduced in \cite{labourie-anosov, guichard-wienhard:anosov}, and we discuss several interesting examples. We then review the construction of their flag domains of discontinuity, based on \cite{guichard-wienhard:anosov, KLP2}.
 

\subsection{Definition and properties}

There are several equivalent definitions of Anosov representations in literature, see \cite{labourie-anosov, guichard-wienhard:anosov, KLP3, GGKW, BPS, KP}. Here we will describe the one that is more suitable for our aims. Let $G$ be a connected semisimple Lie group with finite center and $P$ a parabolic subgroup of $G$ that is conjugate to its opposite parabolic subgroup $P^{\textit{op}}$. Two points $p$ and $q$ in $G/P$ are called \emph{transverse} if there exists $g\in G$ such that $g\Stab_G(p) g^{-1} = P$ and $g\Stab_G(q)g^{-1} = P^{\textit{op}}$.

Let now $\Gamma$ be a finitely generated hyperbolic group with Gromov boundary $\partial_\infty \Gamma$.

\begin{Definition} \label{def - Anosov Representation}
A representation $\rho\co\Gamma \to G$ is $P$-\emph{Anosov} if there exists a continuous, $\rho$-equivariant map 
\[ \xi=\xi_{\rho} \co \partial_\infty \Gamma \longrightarrow G/P\] 
that is
\begin{itemize}
\item \emph{tranverse}, i.e. $\xi_\rho(x)$ and $\xi_\rho(y)$ are tranverse for all $x \neq y \in \partial_\infty \Gamma$;
\item \emph{strongly dynamics preserving}, i.e. for any sequence $(\gamma_n)_{n\in \N} \in \Gamma^\N$ with $\gamma_n \underset{n\to +\infty}{\longrightarrow} \gamma_+ \in \partial_\infty \Gamma$ and $\gamma_n^{-1} \underset{n\to +\infty}{\longrightarrow} \gamma_- \in \partial_\infty \Gamma$, 
\[\rho(\gamma_n)\cdot p \underset{n\to +\infty}{\longrightarrow} \xi_\rho(\gamma_+)\]
for all $p\in G/P$ transverse to $\xi_\rho(\gamma_-)$.
\end{itemize}
A subgroup $\Gamma$ of $G$ is called Anosov if it is hyperbolic and the inclusion $\Gamma \hookrightarrow G$ is Anosov with respect to some proper parabolic subgroup $P$ of $G$.
\end{Definition}

We denote by $\mathrm{Anosov}_P(\Gamma,G)$ the subset of $\mathrm{Hom}(\Gamma,G)$ consisting of $P$-Anosov representations. Note that $P$-Anosov representations are discrete and have finite kernel.  In this paper we will only work with groups $\Gamma$ that are torsion-free. For such groups, $P$-Anosov representations are thus discrete and faithful.

One of the most important properties of Anosov representations is their structural stablility, i.e. $\mathrm{Anosov}_P(\Gamma,G)$ is open in $\mathrm{Hom}(\Gamma,G)$. Structural stability gives a way to construct several Anosov representations as small deformations of a fixed Anosov representation. 
This is a major source of examples, as we will discuss in Section \ref{sub:examples anosov lattices}.

Another important property of $P$-Anosov representations is that they admit cocompact domains of discontinuity in boundaries of $G$, i.e. in homogeneous spaces $G/Q$, where $Q$ is a proper parabolic subgroup of $G$, possibly different from $P$. We will discuss this property in Section \ref{sub:dod}.

\subsection{Construction of Anosov representations via deformation} 
  
\label{sub:examples anosov lattices}

Let us fix a connected semisimple Lie group $H$ of real rank $1$ with finite center, and let $K \subset H$ be its maximal compact subgroup. The symmetric space $S_H = H/K$ has strictly negative sectional curvature and is thus Gromov hyperbolic. Recall that a \emph{uniform lattice} $\Gamma < H$ is a discrete cocompact subgroup of $H$. Any such lattice is quasi-isometric to $S_H$ and is thus a hyperbolic group. Moreover, $H$ has a unique conjugacy class of parabolic subgroups $P_H$. By Guichard--Wienhard \cite[Thm 5.15]{guichard-wienhard:anosov}, $\Gamma$ is a $P_H$-Anosov subgroup of $H$. We will always assume that $\Gamma$ is torsion-free, which is always virtually true by Selberg's lemma.

\begin{Remark}
Note that the Anosov subgroups of a real rank $1$ Lie group $H$ are precisely its quasi-isometrically embedded (equivalently: quasi-convex, or convex-cocompact) subgroups.
\end{Remark}

An important case is when $H$ is a compact extension
\footnote{For example, $H$ can be $\SL(2,\R)$, or $\SL(2,\R) \times \OG(n)$.} 
of
$\PSL(2,\R)$ (i.e. $H$ admits a surjective morphism to $\PSL(2,\R)$ with compact kernel). In that case,  $S_H$ is the hyperbolic plane $\HH^2$, and a torsion-free cocompact lattice $\Gamma$ in $H$ is a \emph{surface group}, i.e. $\Gamma = \pi_1(\Sigma)$, where $\Sigma$ is a closed orientable surface of genus $g \geq 2$. 
A representation $\rho_0:\pi_1(\Sigma) \rightarrow \PSL(2,\R)$ is called \emph{Fuchsian} if it is discrete and faithful (in which case $\rho_0(\pi_1(\Sigma)) \backslash \HH^2$ is a closed hyperbolic surface diffeomorphic to $\Sigma$). Similarly, a discrete and faithful representation into a compact extension $H$ of $\PSL(2,\R)$ will be called a \emph{twisted Fuchsian} representation. It is the case if and only if its projection to $\PSL(2,\R)$ is Fuchsian.

Other interesting cases arise when $H$ is (a compact extension of) $\PO_0(1,n)$ or $\PU(1,n)$, in which cases the symmetric space $S_H$ is respectively the real hyperbolic space $\HH^n = \HH^n_{\R}$ and the complex hyperbolic space $\HH^n_{\C}$. The group $\Gamma$ is then the fundamental group of a closed real hyperbolic or complex hyperbolic manifold. The other Lie groups of real rank $1$ (namely, $\Sp(1,n)$ and $\mathrm F_4^{-20}$) are less interesting for this paper since their lattices are superrigid (see below). Still, our Theorem \ref{thm_intro_main} applies also to them.

Let us fix a uniform torsion-free lattice $\Gamma \subset H$. We choose an embedding $\iota:H \rightarrow G$, where $G$ is a connected semisimple Lie group $G$ with finite center. By Guichard-Wienhard \cite[Prop. 4.7]{guichard-wienhard:anosov}, the representation $\iota \circ  \rho_0 : \Gamma \rightarrow G$ is $P$-Anosov for certain parabolic subgroups $P$ of $G$ described in \emph{loc. cit}.
We will call such a representation an $\iota$-\emph{lattice representation} of $\Gamma$ in $G$. 

When $H$ is $\PSL(2,\R)$, the representation $\iota \circ  \rho_0 : \pi_1(\Sigma) \rightarrow G$ will be called an $\iota$-\emph{Fuchsian representation} in $G$. Similarly, when $H$ is a compact extension of $\PSL(2,\R)$, $\iota \circ  \rho_0$ will be called a \emph{twisted} $\iota$-\emph{Fuchsian representation} in $G$.

Using the property of structural stability, we can deform the representation $\iota \circ \rho_0$, obtaining an open subset of $\mathrm{Hom}(\Gamma,G)$ entirely consisting of $P$-Anosov representations of $\Gamma$ in $G$. In the following we denote by $\mathrm{Anosov}_{P,\iota,\rho_0}(\Gamma,G)$ the connected component of $\mathrm{Anosov}_{P}(\Gamma,G)$ that contains the representation $\iota \circ \rho_0$. We will say that a representation of $\Gamma$ is a \emph{$P$-Anosov deformation of a lattice representation} if it belongs to one of the connected components $\mathrm{Anosov}_{P,\iota,\rho_0}(\Gamma,G)$. In the special case when $H$ is a compact extension of $\PSL(2,\R)$, such representations will be called \emph{$P$-Anosov  deformations of a twisted $\iota$-Fuchsian representation}. 


The main source of examples of interesting deformation spaces of Anosov representations is the case when $H$ is a compact extension of $PSL(2,\R)$, i.e. the case of surface groups. These examples are discussed in Section \ref{sub:examples anosov surfaces}. In the case of uniform lattices in $\PO_0(n,1)$, i.e. fundamental groups of closed hyperbolic $n$-manifolds, we know several constructions of interesting deformations. For example, when $\iota:\PO_0(n,1) \rightarrow \mathrm{PSL}(n+1,\R)$ is the canonical embedding, then every deformation (not just a small deformation) of $\iota \circ  \rho_0$ is an Anosov representation, see Benoist \cite{Benoist3}. The same behaviour is expected when $\iota:\PO_0(n,1) \rightarrow \PO_0(n,p)$ is the canonical embedding, see the discussion in \cite{Wienhard_ICM}. When $p=2$ this was proved by Barbot and Merigot \cite{Barbot-15, barbot-merigot}. Our Theorem~\ref{thm_intro_main} applies to all these deformations, as well as to small deformations of such lattices in the complexification of the Lie group.

Lattices in $\PU(n,1)$ exhibit more rigid behaviour, see \cite{Klingler} and references therein. Still, some of them admit interesting Zariski dense deformations into higher rank Lie groups, but very few examples are known. 
When $H \simeq \Sp(1,n)$ or $\mathrm F_4^{-20}$, by a theorem of Corlette \cite{Corlette-rigidity}, $\Gamma$ is superrigid. In particular, there are no non-trivial deformations of $\iota \circ  \rho_0 : \Gamma \rightarrow G$.

\begin{Remark}
All the arguments in this Section \ref{sub:examples anosov lattices} are more general than the way we presented them. The hypothesis that $\Gamma$ is torsion-free is not really needed, and we can also replace the assumption that $\Gamma$ is a uniform lattice in $H$ with the more general assumption that $\Gamma$ is a convex-cocompact subgroup of $H$. Also in this higher generality, embeddings of $H$ in other groups $G$ allow to construct  open subsets of Anosov representations in $G$. In our discussion, however, we restricted our attention to torsion-free uniform lattices for additional clarity and because our Theorem \ref{thm_intro_main} below works in this special case. 
\end{Remark}

\subsection{Anosov representations of surface groups} \label{sub:examples anosov surfaces}

The case of surface groups is the one that is best understood. When $H$ is a compact extension of $\mathrm{PSL}(2,\R)$, Lie theory gives a classification of all representations $\iota: H \rightarrow G$, for a simple group $G$. Most of the time, twisted $\iota$-Fuchsian representations into $G$ admit small deformations with Zariski dense image. 

In special cases, for particular twisted $\iota$-Fuchsian representations into $G$, all deformations (not just small ones) are Anosov. This phenomenon gives rise to the so called \emph{higher rank Teichm\"uller components}, defined as connected components of the representation variety $\mathrm{Hom}(\pi_1(S),G)/G$ that consist entirely of discrete and faithful representations. They generalize many aspects of classical Teichm\"uller spaces, which can be seen as the spaces of (equivalence classes of) marked hyperbolic structures on a surface, or equivalently as a subset of $\mathrm{Hom}(\pi_1(S),\mathrm{PSL}(2,\R))/\mathrm{PSL}(2,\R)$.
 
There are four families of higher rank Teichm\"uller spaces (see Guichard--Wienhard \cite{GW_pos2}), and most of them are deformations of twisted $\iota$-Fuchsian representations. The first family are \emph{Hitchin representations}, introduced by Hitchin \cite{Hitchin}. In fact Labourie's original motivation for defining Anosov representations in \cite{labourie-anosov} was showing that Hitchin representations form higher rank Teichm\"uller spaces. Hitchin representations are defined when $G$ is a split real simple Lie group. Then $G$ admits a special conjugacy class of representations $\iota_0:\SL(2,\R) \rightarrow G$ called the \emph{principal representation}. For this choice, the representation $\iota_0 \circ  \rho_0$ is called a \emph{principal Fuchsian representation} in $G$, and it is Anosov with respect to the minimal parabolic subgroups $P_{min}$. Hitchin components are the connected components containing a \emph{principal Fuchsian representation}. In particular, any Hitchin representation is a deformation of a principal Fuchsian representation. 
The second family are \emph{Maximal representations}, which are defined when $G$ is a real simple Lie group of Hermitian type \cite{BIW}. Maximal representations are in general only Anosov with respect to a particular maximal parabolic subgroup. Most maximal representations are deformations of twisted $\iota$-Fuchsian representations, but for $G$ locally isomorphic to $\Sp(4,\R)$ there exist connected components in the space of maximal representations where every representation is Zariski dense \cite{Gothen, GW_top, BGG_sp4}. The other two families of higher Teichm\"uller spaces arise from the notion of $\Theta$-positivity introduced in \cite{GW_pos, GW_pos2, GLW_pos}, which leads to the notion of positive representations. Hitchin representations and maximal representations are positive representations, but there are two further families of Lie groups admitting positive representations, Lie groups locally isomorphic to $\SO(p,q)$, as well as an exceptional family. 
With a positive structure comes again a special representation $\iota:\SL(2,\R) \rightarrow G$, and deformations of twisted $\iota$-Fuchsian representations account for most connected components of \emph{positive representations}, exceptions occur only for $\SO(p,p+1)$.  See \cite{GW_pos}, Guichard-Labourie-Wienhard \cite{GLW_pos}, Collier\cite{Collier}, Aparicio-Arroyo--Bradlow--Collier--Garcia-Prada--Gothen--Oliveira \cite{ABCGGO}, Bradlow--Collier--Garcia-Prada--Gothen-Oliveira \cite{BCGGO}, Guichard--Labourie--Wienhard \cite{GLW_pos}, Beyrer--Pozzetti \cite{BP} for more details on positive representations.
%
%
%

Given a representation $\rho\co\pi_1(\Sigma) \to G$ in a higher Teichm\"uller space, we can embed $G$ into its complexification $G_\C$. If $\rho$ is Anosov with respect to a parabolic subgroup $P$, the composition will be Anosov with respect to the parabolic $P_\C<G_\C$. In the complex group not every deformation will be discrete and faithful, but we can consider the space of Anosov representations $\mathrm{Anosov}_{P_\C}(\pi_1(\Sigma),G_\C)$ and the connected component of this space containing $\rho: \pi_1(\Sigma) \to G < G_\C$. This generalizes the notion of quasi-Fuchsian representation into $\PSL(2,\C)$ to this higher rank setting. Of particular interest to us will be the connected component of $\mathrm{Anosov}_{P_\C}(\pi_1(\Sigma),G_\C)$ which contains the principal Fuchsian representation $\iota_0 \circ  \rho_0$. We call this set the \emph{quasi-Hitchin space} and representations therein \emph{quasi-Hitchin representations}. Theorem~\ref{thm_intro_main} applies in particular to quasi-Hitchin representations, and Theorem~\ref{fiber_smooth} focuses on quasi-Hitchin representations for $G_\C = \Sp(4,\C)$
%

\subsection{Domains of discontinuity} \label{sub:dod}

A $P$-Anosov representation $\rho\co\Gamma \to G$ acts on all homogeneous spaces $G/Q$, where $Q$ is a proper parabolic subgroup. The theory of \emph{domains of discontinuity}, introduced by Guichard--Wienhard \cite{guichard-wienhard:anosov} and further developed by Kapovich--Leeb--Porti \cite{KLP2}, gives conditions for the existence of a $\rho$-invariant open subset $\Omega \subset G/Q$ where the action is properly discontinuous and/or cocompact. We sketch very briefly this construction here and refer the reader to  \cite{KLP2} for details.

The action of $P$ on $G/Q$ has finitely many orbits which are labelled by elements of $W_P \backslash W/W_Q$, where $W$ is the Weyl group of $G$ and $W_P, W_Q$ are the subgroups corresponding to $P$ and $Q$. A subset $I$ of $W_P \backslash W/W_Q$ corresponds to a $P$-invariant subset $K_I$ of $G/Q$ (consisting of the union of the orbits labelled by elements of $I$). The set $K_I$ is closed if and only if $I$ is an \emph{ideal} for the Bruhat order on W. 

Given $x = gP\in G/P$, set $K_I(x) = g K_I$ (this is well-defined since $K_I$ is $P$-invariant). The \emph{$I$-thickening} of a subset $A\subset(G/P)$ is the set $K_I(A) = \bigcup_{x\in A} K_I(x)$. Finally, an ideal is called \emph{balanced} if $I \cap -I = \emptyset$ and $I\cup -I = W_P\backslash W /W_Q$.

Now, let $\Gamma$ be a hyperbolic group, $\rho\co \Gamma \to G$ a $P$-Anosov representation and $\xi_\rho\co\partial_\infty \Gamma \to G/P$ the associated boundary map.

\begin{Theorem}[Kapovich--Leeb--Porti, \cite{KLP2}] \label{thm - Domains of discontinuity}
If $I \subset W_P\backslash W /W_Q$ is a balanced ideal, then $\Gamma$ acts properly discontinuously and cocompactly on the domain
\[\Omega_{\rho,I} = (G/Q) \ \setminus \ K_I(\xi_\rho(\partial_\infty \Gamma))~.\]
\end{Theorem}

\begin{Remark} 
If the ideal satisfies $I\cup -I = W_P\backslash W /W_Q$, the construction still gives rise to a domain of discontinuity. But then the action of $\Gamma$ on $\Omega_{\rho,I}$ is not necessarily cocompact. 
\end{Remark}

\begin{Remark}
Note that if the image of the boundary map is preserved (as a set) by a subgroup $H<G$, then this subgroup also naturally acts on $\Omega_{\rho,I}$. This is a key property of the domains of discontinuity $\Omega_{\rho,I}$ we are going to use. 
\end{Remark}

\begin{Remark}
The construction of domains of discontinuity was further generalized by Stecker--Treib \cite{ST}, who extended it to the case where $Q$ is an \emph{oriented parabolic subgroup}, i.e. a subgroup of $G$ lying between a parabolic subgroup and its identity component. The corresponding homogeneous space $G/Q$ is called an \emph{oriented flag variety}. Stecker--Treib \cite{ST} give conditions for the existence of (possibly cocompact) domains of discontinuity on $G/Q$. This generalization is interesting because some new cocompact domains of discontinuity arise that are not lifts of domains of discontinuity in the corresponding unoriented flag varieties. We refer the reader to Kapovich-Leeb-Porti \cite{KLP2} and Stecker-Treib \cite{ST} for more details. 
\end{Remark}

In order to illustrate the theory, let us now describe the example that will be studied in detail in Part \ref{part2} of this paper. Consider the case where the group $G$ is $\Sp(2n,\K)$, where $\K$ can can be $\R$ of $\C$,  and the parabolic subgroup $P$ is the stabilizer of a point in $\KP^{2n-1}$. In this case, $G/P=\KP^{2n-1}$. Every $P$-Anosov representation $\rho\co \Gamma \to \Sp(2n,\K)$ has an associated $\rho$-equivariant map
\[\xi_\rho \co \partial_{\infty}\Gamma \to \KP^{2n-1}\,.\]

We now consider as second parabolic subgroup $Q$ the stabilizer of a Lagrangian subspace in $\K^{2n}$. Then $G/Q$ is the \emph{Lagrangian Grassmannian} $\Lag(\K^{2n})$, i.e. the space of all the Lagrangian subspaces of $\K^{2n}$. The action of $P$ on $G/Q$ has only two orbits: a closed orbit consisting of Lagrangian subspaces containing the line fixed by $P$, and its complement which is open. In this case, $W_P \backslash W/W_Q$ has only two elements and admits a unique non-trivial ideal $I$, for which $K_I$
is the closed $P$-orbit. This ideal is balanced. 

For each line $\ell \in \KP^{2n-1}$ we have
\[K_\ell = K_I(\ell) = \{W \in \mathrm{Lag}(\K^{2n}) \mid \ell \subset W\} \subset \Lag(\K^{2n})~,\]
and we define the subset 
\[K_{\rho,I} = K_I(\xi(\partial_\infty \Gamma)) = \bigcup_{t \in \partial_{\infty}\pi_1(\Sigma)}K_{\xi(t)} \subset \Lag(\K^{2n})~.\]
Guichard and Wienhard \cite{guichard-wienhard:anosov} showed that the complement
\[\Omega_{\rho,I} = \mathrm{Lag}(\K^{2n}) \ \setminus\ K_{\rho,I}\] 
is a cocompact domain of discontinuity for $\rho$. The manifold 
\[M_{\rho,I} = \rho(\Gamma) \backslash \Omega_{\rho,I}\] 
is a closed manifold endowed with a geometric structure modelled on the parabolic geometry $(G, G/Q) = \left(\Sp(2n,\K), \Lag(\K^{2n})\right)$. 
Determining the topology of this quotient manifold (and more general such constructions) is one of the main focus of this paper.

\subsection{Deformations}

Consider one of our spaces $\mathcal{A} = \mathrm{Anosov}_{P,\iota,\rho_0}(\Gamma,G)$, defined in Section \ref{sub:examples anosov lattices}. Let $Q$ be a parabolic subgroup and $I$ a balanced ideal of $W_P \backslash W /W_Q$. For every representation $\rho \in \mathcal{A}$, we obtain a closed manifold $M_{\rho,I} = \rho(\Gamma) \backslash \Omega_{\rho,I}$ endowed with a geometric structure locally modelled on the parabolic geometry $(G, G/Q)$, and whose holonomy factors through
\footnote{If $\Omega_\rho$ is not simply connected, then $\Gamma$ is only a quotient of $\pi_1(M_\rho)$.}  
the representation $\rho$.

\begin{Theorem}[Guichard--Wienhard, \cite{guichard-wienhard:anosov}] \label{thm:ConstantTopology}
Let $\rho$ be a $P$-Anosov representation of a hyperbolic group $\Gamma$ into a semisimple Lie group $G$ and let $\rho'$ be a $P$-Anosov deformation of $\rho$. Then for any parabolic subgroup $Q$ of $G$ and any balanced ideal $I$ of $W_P\backslash W/W_Q$, there exists a smooth $(\rho,\rho')$-equivariant diffeomorphism from $\Omega_{\rho,I}$ to $\Omega_{\rho',I}$. In particular, $M_{\rho,I}$ and $M_{\rho',I}$ are diffeomorphic.
\end{Theorem}

\begin{Remark}
The theorem also applies to the quotients of domains of discontinuity constructed by Stecker--Treib in oriented flag varieties.  More generally, it essentially follows from Ehresmann's fibration theorem that a smooth family of closed $(G,X)$-manifolds is locally topologically trivial.
\end{Remark}

For a $\rho \in \mathcal{A} = \mathrm{Anosov}_{P,\iota,\rho_0}(\Gamma,G)$, the topology of $M_{\rho,I}$ does not depend on~$\rho$, hence we can denote this smooth manifold by $M_{\rho_0,\iota,I}$. Thus, the space $\mathcal{A}$ can be seen as a deformation space for a family of $(G,G/Q)$-structures on a the fixed closed manifold $M_{\rho_0,\iota,I}$. This is particularly interesting for higher rank Teichm\"uller spaces, because it gives a nice geometric interpretation of these spaces. It is also interesting for the theory of geometric structures on manifolds, because it gives several interesting examples of closed manifolds with a large deformation space of geometric structures.


\section{Topology of the quotient}  \label{sec:topology_of}

\subsection{General statement}
We can now rephrase Theorem \ref{thm_intro_main}, which describes the topology of $M_{\rho_0,\iota,I}$ constructed from an Anosov deformation of an $\iota$-lattice representation. 

Let us fix a connected semisimple Lie group $H$ of real rank $1$ with finite center, a uniform torsion-free lattice $\Gamma \subset H$ and a representation $\iota$ of $H$ into some connected semisimple Lie group $G$ with finite center. Denote by $\rho_0$ the inclusion of $\Gamma$ into $H$, let $P$ be a parabolic subgroup of $G$ such that $\iota \circ \rho_0$ is $P$-Anosov, and let $\rho$ be a $P$-Anosov deformation of $\iota\circ \rho_0$. Finally, let ${S_H}$ denote the symmetric space of $H$.

\begin{Theorem} \label{thm:TopologyQuotient}
For every parabolic subgroup $Q$ of $G$ and every balanced ideal $I$ of $W_P \backslash W/W_Q$, the domain $\Omega_{\rho,I}$ is a smooth $\Gamma$-equivariant fiber bundle over the symmetric space ${S_H}$, with fiber a closed manifold $\mathfrak F$. In particular, $\Omega_{\rho,I}$ deformation retracts to $\mathfrak F$ and the manifold $M_{\rho_0,\iota,I} = \Gamma \backslash \Omega_{\rho,I}$ is a fiber bundle over the locally symmetric space $\Gamma \backslash {S_H}$ with fiber $\mathfrak F$.
\end{Theorem}

In fact, one can say a bit more on the structure of this bundle. Let $K$ denote a maximal compact subgroup of $H$, so that ${S_H}=H/K$. Recall that, given a principal $K$-bundle $B$ over a manifold $T$ and a smooth action of $K$ on a manifold $\mathfrak F$, the $\mathfrak F$-bundle associated to $B$ is the quotient of $B\times \mathfrak F$ by the diagonal action of $K$. The projection to the first factor gives it the structure of a fiber bundle over $T$ with fiber $\mathfrak F$.

\begin{Theorem} \label{thm:BundleStructureQuotient}
In the setting of Theorem \ref{thm:TopologyQuotient}, the manifold $\mathfrak F$ admits a smooth action of the compact subgroup $K$ and the manifold $M_{\rho_0,\iota,I}$ is the $\mathfrak F$-bundle over $\Gamma \backslash {S_H}$ associated to the principal $K$-bundle
\[\Gamma \backslash H \to \Gamma \backslash {S_H}\,.\]
\end{Theorem}

In the previous theorem, the bundle $\Gamma \backslash H \to \Gamma \backslash {S_H}$ must be thought of as an explicit object that depends only on the lattice $\Gamma$. For example, when $H = PO_0(n,1)$,  $\Gamma \backslash {S_H}$ is a closed hyperbolic manifold, $\Gamma$ is its fundamental group, and the bundle $\Gamma \backslash H \to \Gamma \backslash {S_H}$ is its frame bundle.
 
In the special case when $H$ is a compact extension of $\PSL(2,\R)$, these theorems take an even more explicit form. In this case, $S_H = \mathbb{H}^2$ is the hyperbolic plane, the group  $\Gamma = \pi_1(\Sigma)$ is a surface group, and $\Gamma \backslash \mathbb{H}^2 = \Sigma$ is the surface. The principal bundle   
\[\rho_0 \backslash H \to \Sigma\]
depends on the extension $H$. For example, when $H = \PSL(2,\R)$, this bundle is a circle bundle isomorphic to the unit tangent bundle of $\Sigma$, i.e. a circle bundle with Euler class $2g-2$. When $H = \SL(2,\R)$, this bundle is the double cover of the unit tangent bundle of $\Sigma$, i.e. a circle bundle with Euler class $g-1$. For all the interesting groups $H$, it is possible to understand this bundle explicitly. We will now restate the previous theorems in the case when $H = \SL(2,\R)$.   
  
\begin{Corollary}  \label{cor:TopologyBundleSurfaces}
Let $P$ be a parabolic subgroup of $G$, $Q$ another parabolic subgroup and $I$ a balanced ideal of $W_Q\backslash W/W_P$.  Let $\rho$ be a $P$-Anosov deformation of an $\iota$-Fuchsian representation of a surface group $\pi_1(\Sigma)$. Then
\begin{itemize}
\item $\Omega_{\rho,I}$ retracts to a closed submanifold $F$ of codimension $2$ carrying a smooth circle action.
\item The quotient $\rho(\pi_1(\Sigma)) \backslash \Omega_{\rho,I}$ is diffeomorphic to a fiber bundle over $\Sigma$ with fiber $F$. This is the $F$-bundle associated to the principal circle bundle of Euler class $g-1$ over $\Sigma$.
\end{itemize}
\end{Corollary}

One of the main applications of this corollary is for (quasi)-Hitchin representations. Recall that Hitchin representations are deformations of $\iota_0 \circ \rho_0$ where $\rho_0: \pi_1(\Sigma)\to \SL(2,\R)$ is a Fuchsian representation and $\iota_0$ is the principal representation of $\SL(2,\R)$ into a real split semisimple Lie group $G$, and quasi-Hitchin representations are their $P_{\min}$-Anosov deformations into its complexification $G_\C$. 

We can also apply Theorems \ref{thm:TopologyQuotient} and \ref{thm:BundleStructureQuotient} to the positive representations that are Anosov deformations of twisted $\iota$-Fuchsian representations. As discussed in Section \ref{sub:examples anosov surfaces}, almost all the positive representations in the classical groups are of this type, with the only exception of the exceptional components in $\Sp(4,\R)$ and $\SO(p,p+1)$. In order to apply our results to positive representations, we need to consider the group $H = \SL(2,\R) \times C$ for a certain compact subgroup $C$. The statement is similar to Corollary \ref{cor:TopologyBundleSurfaces}, except that the structure group of the bundle is now $\SO(2)\times C$. The invariants that characterize the bundle are the Euler class $g-1$ and the first and the characteristic classes of the $C$ component of $\rho_0$.

\subsection{Proof of the theorems} \label{sub:proof}

A key hypothesis in Theorem \ref{thm:TopologyQuotient} is the assumption that $\rho$ is a $P$-Anosov deformation of an $\iota$-Fuchsian representation $\iota\circ \rho_0$. Indeed, by Guichard--Wienhard's Theorem \ref{thm:ConstantTopology}, the topology of $M_{\rho,I}$ does not change, and we only have to determine it for $\rho = \iota\circ \rho_0$. 
The key result for the proof is thus the following.

\begin{Lemma} \label{lem:proper action}
Let $H$ be a semisimple Lie group with finite center, and $K \subset H$ be its maximal compact subgroup. Let $X$ be a manifold with a proper action of $H$. Then there exists an $H$-equivariant smooth fibfration 
\[p:X \rightarrow S_H~,\]
where $S_H$ denotes the symmetric space of $H$.
\end{Lemma}

For the proof, we use the fact that the symmetric space $S_H = H/K$ has non-positive curvature.
We need the notion of barycenter: Given a finite measure of compact support $\nu$ on $S_H$, we consider the function
\[b:S_H \rightarrow \R \]
defined by
\[b(y) = \int d(y,z)^2 \mathrm d \nu(z)\,.\] 
Notice that the squared distance function is smooth, hence $b$ is also smooth. Since $S_H$ has non-positive curvature, the distance function is convex (see \cite[Thm 1.3]{BGS}), hence the squared distance function is strongly convex. 
This implies that the function $b$ is strongly convex, 
%
%
and since it is also proper, it has a unique  critical point, which is a global minimum. Moreover, the Hessian of $b$ at the global minimum is positive definite.  

The \emph{barycenter} $\mathrm{Bar}\{\nu\}$ of $\nu$ is defined to be the unique critical point of $b$.

\begin{proof}[Proof of Lemma \ref{lem:proper action}]
We choose a torsion-free uniform lattice  $\Gamma$ in $H$, this always exists, see for example Borel and Harish-Chandra \cite{BHC}. Then $\Gamma$ acts freely, properly discontinuously and cocompactly on $X$, hence the quotient $\Gamma \backslash X$ is a closed manifold. Then $\Gamma$ is isomorphic to the quotient $\pi_1(\Gamma \backslash X)/\pi_1(X)$, and we have a homomorphism $\psi\co\pi_1(\Gamma\backslash X) \to \Gamma = \pi_1(\Gamma \backslash S_H)$. Since $S_H$ is contractible there exists a map $\Gamma \backslash X \to \backslash S_H$ inducing $\psi$ (for the details, \cite[Prop. 13]{Aless_MastThesis}). A priori this map is only continuous, but since smooth maps are dense in the space of continuous maps between compact manifolds, we can assume that the map  is smooth. Lifting this map, we obtain a smooth $\Gamma$-equivariant map
\[f\co X \to S_H\,.\]

This allows us to define, for all $x\in X$, a smooth map $F^x\co H \to S_H$ 
by 
\[F^x(g) = g\cdot f(g^{-1}\cdot x)\,.\] 
Since $f$ is $\Gamma$-equivariant, we have that $F^x(g\gamma)= F^x(g)$ for all $\gamma\in \Gamma$. Hence $F^x$ descends to a map $H/\Gamma \to S_H,$ that with abuse of notation we keep calling $F^x$.

Note that we have 
\begin{equation}\label{eq:Equivariance F(x,g)}
F^{h\cdot x}(g) = g \cdot f(g^{-1} h \cdot x) = h \cdot F^{x}(h^{-1} g)~.
\end{equation}

We can finally define a map $\bar f\co X \to S_H$ as follows:
\[\bar f(x) = \mathrm{Bar} \{F^x_* \mu\}\,,\]
where $\mu$ is the Haar measure on $H/\Gamma$, $F^x_*\mu$ its push-forward by $F^x$, and $\mathrm{Bar}$ is the barycenter of a finite measure of compact support, as defined above.

We claim that $\bar f$ is $H$-equivariant. Indeed, for $x\in X$ and $g\in H$, we have
\begin{eqnarray*}
\bar f(gx) &=& \mathrm{Bar} \{ F^{gx}_*\mu\} \\
&=& \mathrm{Bar}\{g_* F^x_* g^{-1}_* \mu\} \quad \textrm{by \eqref{eq:Equivariance F(x,g)}}\\
&=& \mathrm{Bar}\{g_* F^x_* \mu \} \quad \textrm{by left invariance of the Haar measure}\\
&=& g\cdot \mathrm{Bar}\{F^x_* \mu \} \quad \textrm{by equivariance of the barycenter map}\\
&=& g\cdot \bar f(x)~.
\end{eqnarray*}

The rest follows from the two lemmas below. Lemma \ref{lem:smooth} guarantees that the map $\bar{f}$ is smooth, and Lemma \ref{lem:fiber bundle} shows that it is an Ehresmann fibration.
%
\end{proof}

\begin{Lemma} \label{lem:smooth}
The map $\bar f$, constructed in the proof of Lemma \ref{lem:proper action}, is smooth.
\end{Lemma}
\begin{proof}
For every $x\in X$, $\bar{f}(x)$ is the unique critical point of the function
\[b(y) = \int_{S_H} d(y,z)^2 \mathrm d \nu(z)\,,\]
where the measure $\nu$ is the push-forward $\nu = F^x_* \mu$. By the change-of-variable formula, this can be written as
\[b(y) = \int_{\Gamma \backslash H} d(y,F^x(h))^2 \mathrm d \mu(h) = \int_{\Gamma \backslash H} d(y,h f(h^{-1}x))^2 \mathrm d \mu(h)\,.\]
In this setting, the function $b$ depends on the parameter $x$, to make this more explicit, we write it as $b(x,y)$, a smooth function of two variables $x \in X$ and $y \in S_H$. We consider the differential of $b$ with reference to $y$:
\[\beta(x,y) = d_y b(x,y) : X \times S_H \rightarrow T^* S_H\,. \]
Let's fix an $x_0 \in X$ and the corresponding $y_0 = \bar{f}(x_0) \in S_H$. We choose local coordinates on a small neighborhood $U$ of $y_0$ in $S_H$. This trivializes the cotangent bundle on $U$: $T^* U \simeq U \times\R^k$, where $k = \dim(S_H)$. Let $\pi_2: T^* U \rightarrow \R^k$ denote the projection onto the second factor. Now we consider the composition 
\[\pi_2 \circ \beta(x,y) = d_y b(x,y) : X \times U \rightarrow \R^k \,. \]
Now, for all $x$ close enough to $x_0$, the pairs $(x,\bar{f}(x))$ are precisely the solutions to the equation
\[\pi_2 \circ \beta(x,y) = 0\,.\] 
We can now apply the implicit function theorem to the function $\pi_2 \circ \beta(x,y)$. The differential of this function is non-degenerate because the Hessian of the strongly convex function $b(y)$ is positive definite at its critical point. The implicit function theorem guarantees that $\bar{f}$ is smooth.
\end{proof}



\begin{Lemma}  \label{lem:fiber bundle}
Let $H$ be a Lie group, and let $X, Y$ be spaces with $H$-actions, where the action on $Y$ is transitive. Denote by $L$ the stabilizer in $H$ of a point of $Y$, and identify $Y$ with $H/L$. Then, every $H$-equivariant map $\phi:X \rightarrow Y$ is a fiber bundle, with structure group given by an action of $L$ on the fiber. The bundle is associated, via a change of fiber, to the principal $L$-bundle $H \rightarrow H/L$. 
\end{Lemma}
\begin{proof}
First note that, by homogeneity of $Y$, the map $\phi$ needs to be onto. We will construct a local trivialization around every point $y \in Y$. We choose the subgroup $L$ as the stabilizer of $y$. Hence, $L$ is acting on the fiber $F = \phi^{-1}(y)$. 
Let $U$ be a neighborhood of $y$ in $Y$ that trivializes the bundle $H \rightarrow H/L$. The trivialization is a map $t:U \times L \rightarrow H$.  

A trivialization of $\phi$ over $U$ is given by the map
\[ T: U \times F \ni (u,f) \rightarrow t(u,e)f \in X\,,\]
where $e$ is the identity of $H$. Clearly, $\phi(T(u,f)) = u$, because $t(u,e)$ sends $y$ to $u$. 
The map $T$ is 1-1 because if $t(u,e)f = t(u',e)f'$, then $u = u'$  because $\phi(T(u,f))=u$, and then by multiplying by $t(u,e)^{-1}$ we see that $f=f'$. We can also see that the map $T$ is onto $\phi^{-1}(U)$, because given $x \in \phi^{-1}(U)$, let $f = t(\phi(x),e)^{-1}x \in F$, and then $x = T(\phi(x),f)$. 

The construction above shows that every atlas for the bundle $H \rightarrow H/L$ induces an atlas for the bundle $\phi: X \rightarrow Y$. It is easy to check that the two atlases have the same transition functions, hence the two bundles are associated.
\end{proof}

The following Proposition~\ref{prop:lattice case} is similar to Theorems~\ref{thm:TopologyQuotient} and \ref{thm:BundleStructureQuotient}, but the difference is that it can be applied to domains of discontinuity that are not necessarily cocompact. Anyway, if the domain is not cocompact, our conclusion only holds for $\iota$-lattice representations, but does not automatically extend to their deformations. After that, we will add the hypothesis that the domain of discontinuity is cocompact and prove the full Theorems~\ref{thm:TopologyQuotient} and \ref{thm:BundleStructureQuotient} for their deformations.





\begin{Proposition} \label{prop:lattice case}
Let $H$ be a connected semisimple Lie group with finite center of real rank $1$, $K \subset H$ a maximal compact subgroup, and $S_H = H/K$ be the symmetric space for $H$. Let $\rho_0:\Gamma \rightarrow H$ be the inclusion of a torsion-free uniform lattice in $H$. Let $G$ be a connected semi-simple Lie group with finite center, and $\iota:H \rightarrow G$ be a representation. Let $P < G$ be a parabolic subgroup of $G$ such that the representation $\rho = \iota \circ \rho_0$ is $P$-Anosov.  

Let $Q$ be a parabolic subgroup of $G$
%
%
%
 and $\Omega_{\rho,I} \subset G/Q$ a domain of discontinuity for $\rho$ constructed from a thickening of a (not necessarily balanced) ideal $I$. Let $M_{\rho,I} = \rho \backslash \Omega_{\rho,I}$ be the quotient manifold.

Then $M_{\rho,I}$ is diffeomorphic to a smooth fiber bundle over $S_\Gamma$. The fiber $F$ of the bundle is homotopically equivalent to the domain $\Omega_\rho$, and carries a $K$-action that gives the bundle a structure of $K$-bundle. The bundle is isomorphic to the $K$-bundle associated to the $K$-principal bundle $\Gamma \backslash H \rightarrow \Gamma \backslash S_H$ via a change of fiber.
\end{Proposition}
\begin{proof}
As we saw in Section \ref{sub:examples anosov lattices}, it was proved in Guichard--Wienhard \cite[Prop. 4.7]{guichard-wienhard:anosov} that the representation $\rho$ is $P$-Anosov for a certain family of parabolic subgroups described there. 
Moreover, since $\Gamma$ is a lattice in $H$, we have that $\partial_\infty \Gamma = H/P_H$, so $H$ acts on $\partial_\infty \Gamma $ and the Anosov limit map is in fact $\iota$-equivariant. This implies that $H$ preserves the domain of discontinuity $\Omega_\rho$, and acts properly on $\Omega_\rho$. 
Applying Lemmas \ref{lem:proper action} and \ref{lem:fiber bundle}, we get a smooth $H$-equivariant fiber bundle map from $\Omega_\rho$ to $S_H$, which factors to a smooth fiber bundle map from $M$ to $S_\Gamma$. 
\end{proof}

\begin{proof}[Proof of Theorems~\ref{thm:TopologyQuotient} and \ref{thm:BundleStructureQuotient}]
In Proposition \ref{prop:lattice case}, we have already proved the theorem for twisted $\iota$-Fuchsian and for lattice representations. Now, since we are assuming that the domain of discontinuity is cocompact, it follows from Theorem~\ref{thm:ConstantTopology} that the topology of $M$ is constant in  $\mathcal{A}$. 
\end{proof}

\part{Quasi-Hitchin representations into $\Sp(4,\C)$} \label{part2}

In the second part of the paper, we will focus on quasi-Hitchin representations into  $G = \Sp(4,\C)$. We fix the principal representation $\iota_0 \co \SL(2, \R) \to G$ and a Fuchsian representation $\rho_0 \co \pi_1(\Sigma) \to \SL(2, \R)$. Recall that, for every parabolic subgroup $P$ of $\Sp(4,\C)$, a $P$-quasi-Hitchin representation is a $P$-Anosov deformation of $\iota_0\circ\rho_0$. Our aim is to determine the topology of the quotient manifold of the cocompact domains of discontinuity for these representations. 

The group $\Sp(4,\C)$ has (up to conjugation) three different proper parabolic subgroups, so there are three flag varieties for us to consider, the projective space $\CP^3$, the Lagrangian Grassmannian $\Lag(\C^4)$, and the full flag variety, which consists of full isotropic flags, i.e. pairs consisting of a line in $\C^4$ and a Lagrangian subspace of $\C^4$ containing that line.
The principal Fuchsian representation $\iota_0\circ\rho_0$ admits four cocompact domains of discontinuity constructed by a balanced thickening: one in the projective space $\CP^3$, one in the Lagrangian Grassmannian $\Lag(\C^4)$ (whose construction is described in Section \ref{sub:dod}), and two in the full isotropic flag variety. The two domains of discontinuity in the full isotropic flags variety are in fact the pull back of the two domains in  $\CP^3$ and in $\Lag(\C^4)$ under the natural projection from the full flag variety to the partial flag varieties, hence they can be understood from a description of the latter two. The domain of discontinuity in $\CP^3$ was described in Alessandrini--Davalo--Li \cite[Corol. 10.2]{AleDavLi}, where it is proved that the quotient manifold $M$ is diffeomorphic to a fiber bundle over the surface $\Sigma$ with fiber $\mathbb{S}^2 \times \mathbb{S}^2$. 

The only cocompact domain of discontinuity that is not yet understood is the one in $\Lag(\C^4)$. The second part of this paper is devoted to the description of this domain and its quotient manifold. In fact, the domain of discontinuity in the Lagrangian Grassmannian is of particular interest because it contains two copies of the symmetric space associated to $\Sp(4,\R)$, the Siegel upper half space, and the Siegel lower half space, see \cite{Wienhard_thurston}. This is very reminiscent of the situation for quasi-Fuchsian representations, and we hope that a good understanding of the domain of discontinuity and its quotient manifold might help to shed some light on possible generalizations of the Bers' double uniformization theorem for quasi-Hitchin representations.

The construction of the domain of discontinuity in $\Lag(\C^4)$, described in Section~\ref{sub:dod}, only uses the fact that the representation is Anosov with respect to $P$, where $P$ is the stabilizer of a point in $\CP^3$. We will thus consider a representation $\rho$ in the quasi-Hitchin space $\mathrm{QHit}_{P}(\Sigma, \Sp(4,\C)) := \mathrm{Anosov}_{P,\iota_0,\rho_0}(\pi_1(\Sigma),\Sp(4,\C))$. 

There is a unique non-trivial ideal $I$, which allows us to define a domain of discontinuity $\Omega_{\rho,I}$, with quotient manifold $M_{\rho,I}$, see Section \ref{sub:dod}. Since, by Theorem~\ref{thm:ConstantTopology}, the topology of the quotient manifold doesn't depend on $\rho$, we can restrict our attention to the case when $\rho = \iota_0\circ \rho_0$. With the representation fixed once for all, we will denote the domain of discontinuity and the quotient manifold simply by $\Omega$ and $M$, instead of $\Omega_{\rho,I}$ and $M_{\rho,I}$.

Our Theorem \ref{thm_intro_main} gives smooth fibrations
\[p\co\Omega \to \HH^2 \;\; \text{ and } \;\;\widehat{p}\co M \to \Sigma.\] 
In this second part of the paper we will study the fiber $\mathfrak{F} = \mathfrak{F}_p$ of these maps, and prove Theorem \ref{fiber_smooth}, which states that $\mathfrak{F}$ is homeomorphic to the $4$-manifold $\C\mathbb{P}^2 \# \overline{\C\mathbb{P}}^2$.

Our strategy will be to describe a new $\rho$-equivariant fibration
\[q\co\Omega \to \HH^2 \;\; \text{ and } \;\;\widehat{q}\co M \to \Sigma.\] 
which is not smooth, but has a more geometric definition and for which the fiber $F = F_q$ is easier to understand. This new fiber $F$ is not a manifold, but it is homotopically equivalent to the domain of discontinuity $\Omega$, hence $F$ is homotopically equivalent to the smooth fiber $\mathfrak{F}$. We will see that $F$ is simply connected, hence $\mathfrak{F}$ is a simply connected $4$-manifold. A smooth simply connected $4$-manifold is determined up to homeomorphism by its homotopy type, hence we can determine $\mathfrak{F}$ by computing the homotopy invariants of $F$.

In order to describe the fibration $q$ we will study the action of $\SL(2,\C)$ on the Lagrangian Grassmannian $\Lag(\C^4)$ induced by the principal representation $\iota_0$. In Section \ref{sl2orbits} we will discuss the $\SL(2,\C)$-orbits in $\mathrm{Lag}(\C^{4})$, and this discussion will allow us to define $q$ at the end of the section. In Section \ref{sub:new_york} we will study the fiber $F$, and in Section \ref{fiber} we will use our results on the topology of $F$ to understand $\mathfrak{F}$. 


\section{$\SL(2,\C)$-orbits of $\mathrm{Lag}(\C^{4})$}  \label{sl2orbits}

In this section, we will study the action of $\SL(2,\C)$ on $\mathrm{Lag}(\C^{4})$ and its orbits. This will allow us to define the projection $q\co\Omega \to \HH^2$ explicitly.
\subsection{Lagrangian subspaces in $\C^{2n}$}

Let $\K \in \{\R, \C\}$ and let $V_{1, \K} = \K_1[X, Y] \cong \K^2$ be the space of homogeneous polynomials of degree one in the variables $X$ and $Y$, endowed with the symplectic form determined by $\omega_{1, \K}(X,Y) = 1$. The induced action of $\Sp(V_{1, \K}, \omega_{1, \K}) \cong \Sp(2, \K) \cong \SL(2, \K)$ on $V_{n, \K} = \mathrm{Sym}^{2n-1}(V_{1, \K}) = \K^{(2n-1)}[X, Y] \cong \K^{2n}$ preserves the symplectic form $\omega_{n, \K} = \mathrm{Sym}^{2n-1}\omega_{1, \K}$, which is given by 
\[ \left\{
  	\begin{array}{ll}
  		\omega_{n, \K}(P_k, P_l) = 0  & \mbox{if } k+l \neq 2n-1 \\
  		\omega_{n, \K}(P_k, P_{2n-1-k}) = (-1)^k \frac{k! (2n-1-k)!}{(2n-1)!}, & 
  	\end{array}
  \right.
\] 
where $P_k = X^{2n-1-k}Y^k$ for $k = 0, \ldots, 2n-1$. These formulae become more explicit in the case where $n=2$, which is the case we are mainly interested in in the following:
\begin{equation} \label{eq:omega}
\left\{
  	\begin{array}{ll}
  		\omega_{2, \K}(X^3, Y^3) = 1 \,,\\
  		\omega_{2, \K}(X^2 Y, X Y^2) = -\frac{1}{3} \,, & 
  	\end{array}
  \right.
\end{equation} 
all other pairings being zero. 
 
Moreover, this induced action defines the (unique) $(2n)$--dimensional irreducible representation 
\[\pi_{2n}\co \mathrm{Sp}(V_{1, \K}, \omega_{1, \K}) \cong \Sp(2, \K) \cong \mathrm{SL}(2,\K) \to \Sp(V_{n, \K}, \omega_{n, \K}) \cong  \Sp(2n, \K)\,.\]
A vector subspace $L \subset V_{n, \K}$ is called \textit{isotropic} if $L \subset L^{\perp_{\omega_{n, \K}}}$, where $L^{\perp_{\omega_{n, \K}}}$ is the orthogonal complement with respect to $\omega_{n, \K}$. An isotropic subspace $L \subset V_{\K}$ is maximal if it has dimension $n$, or equivalently if $L = L^{\perp_{\omega_{\K}}}$. In this case $L$ is called a {\em Lagrangian} subspace. Using the fact that $\omega_{n, \K}$ is skew-symmetric, we can see that all the subspaces of $V_{n, \K}$ of dimension one are isotropic, hence the space of $1$-dimensional isotropic subspaces can be identified with the projective space $\mathbb{P}(V_{n, \K})$. We denote the space of $n$--dimensional isotropic (Lagrangian) subspaces of $V_{n, \K}$ by $\mathrm{Lag}(V_{n,\K})$, and we will call it the \emph{Lagrangian Grassmannian}. We can think of it as a subspace of the Grassmannian of $n$-dimensional subspaces in $V_{n, \K}$.
 
Recall that $\SL(2, \C) \cong \Sp(2, \C)$ acts on $\mathbb{P}(\C^{(n-1)}[X, Y]) \cong \C\mathbb{P}^{n-1}$ by acting on the roots of the polynomials in $\C^{(n-1)}[X, Y]$, and this action naturally defines an action of $\SL(2, \C)$ on $\mathrm{Lag}(V_{n, \K})$. 
  
One last thing we want to recall about this Lagrangian Grassmanian is its topology. There are different ways to describe the topology, but the way we will mostly use in this paper is via the subspace topology inherited from the Grassmanian space, whose topology can be described using the Pl\"ucker map or Pl\"ucker coordinates. In our case the Pl\"ucker map is an embedding of the Grassmanian of $n$--planes in $V_{n, \K}$ into the projectivization of the $n$-th exterior power of $V_{n, \K}$:
  \[\mathrm{Gr}(n, V_{n, \K}) \to \mathbb{P}\left(\Lambda^n(V_{n, \K})\right)\,,\] 
which realizes $\mathrm{Gr}(n,V_{n, \K})$ as an algebraic variety, since the image consists of the intersection of a number of quadrics defined by the Pl\"ucker relations. To write these relations, we need to be more precise. Given $W \in\mathrm{Gr}(n,V_{n, \K})$, the $n$--dimensional subspace spanned by the basis of column vectors $W_1, W_2, \ldots, W_n$ in $V_{n, \K}$, let $\widehat{W}$ be the $(2n) \times n$ matrix of homogeneous coordinates, whose columns are $W_1, W_2, \ldots, W_n$. For any ordered sequence $1 \leq i_1 < i_2 < \cdots < i_n\leq 2n$ of $n$ integers, let $W_{i_1, \ldots, i_n}$ be the determinant of the $n \times n$ matrix given by the rows $i_1, \ldots, i_n$ of $\widehat{W}$. Then, the relation is given by 
$$\sum_{l=1}^{n+1} W_{i_1, \ldots, i_{n-1}, j_l} W_{j_1, \ldots,\hat{j}_l, \ldots j_{n+1}} = 0,$$
for $1 \leq i_1 < i_2 < \cdots < i_{n-1}\leq 2n$ and $1 \leq j_1 < j_2 < \cdots < j_{n+1}\leq 2n$ and where $\hat{j}_l$ denotes the fact that the $j_l$ term is omitted. For example, in the case $n = 2$, which will be the main focus of the article, we will have coordinates $W_{1, 2}$, $W_{1, 3}$, $W_{1, 4}$, $W_{2, 3}$, $W_{2, 4}$ and $W_{3, 4}$, with the relation
\[W_{1, 2}W_{3, 4} - W_{1, 3}W_{2, 4}+W_{1, 4}W_{2, 3}= 0\,.\]
The condition that the space $W$ is Lagrangian gives one additional polynomial equation, so that $\Lag(\C^4)$ is a projective variety of complex dimension $3$.
 
\subsection{$\SL(2,\C)$-orbits of $\mathrm{Lag}(\C^{4})$}
  
From now on, we focus on dimension $4$. So, let $V_{\K} = \K^{(3)}[X, Y] \cong \K^4$ be the symplectic space of homogeneous polynomials of degree $3$ in $X$ and $Y$, equipped with the symplectic form $\omega_{\K} = \omega_{2, \K}$ defined above. We define the {\it Veronese embeddings}
$$ \xi^1_{\C} \co \C\mathbb{P}^1 \to \C\mathbb{P}^3 \;\;\;\;\; \text{ and }\;\;\;\;\;\xi^2_{\C} \co \C\mathbb{P}^1 \to \mathrm{Lag}(\C^4)$$
by 
\begin{equation*}
 \begin{split}
 \xi^1_\C([a:b]) &:= \langle (bX-aY)^3 \rangle\in \C\mathbb{P}^3\\
 \xi^2_\C([a:b]) &:= \langle (bX-aY)^3, (dX-cY)(bX-aY)^2\rangle \in \mathrm{Lag}(\C^4),
 \end{split}
 \end{equation*}
 where $[c:d]$ is any point in $\C\mathbb{P}^1 \setminus \{[a:b]\}$. Let
$$
 \xi^1 = \xi^1_{\R} = \xi^1_{\C} \mid_{\R\mathbb{P}^1} \;\;\;\;\; \text{ and } \;\;\;\;\;\xi^2 = \xi^2_{\R} = \xi^2_{\C} \mid_{\R\mathbb{P}^1}.
$$
 
Recall, from Section \ref{sub:dod}, that for a line $\ell \in \CP^{3}$ we defined  
\[K_\ell = \{W \in \mathrm{Lag}(\C^{4}) \mid \ell \subset W\} \subset \Lag(\C^{4})~.\]
We now introduce the set 
\[K_\C = \bigcup_{t \in \C\mathbb{P}^1} K_{\xi^1_\C(t)} \subset \Lag(\C^{4})~.\]
Using the definition of $\xi^1_\C$, we can see that $K_{\C}$ is the set 
\[K_{\C} = \{W \in \mathrm{Lag}(\C^{4}) \ \mid\ \exists\; p= (b X- a Y)^3 \in W,\; [a:b]\in \C\mathbb{P}^1\}\]
of Lagrangian subspaces that contain a polynomial with a triple complex root $[a:b] \in \C\mathbb{P}^1$.

\begin{Lemma}
$K_{\C}$ is the set of Lagrangians $W \in \mathrm{Lag}(\C^{4})$ with a common root, i.e.
\[K_{\C} = \{W \in \mathrm{Lag}(\C^{4}) \ \mid\ \exists\; [a:b]\in \C\mathbb{P}^1,\; \forall p \in W,\; p(X,Y) = (b X- aY) q(X,Y)\}\,.\] 
\end{Lemma}
\begin{proof}
Let $W \in K_{\C}$. We know that it contains an element with a triple root. By acting with $SL(2,\C)$, we can assume that the triple root is zero, in other words that $X^3 \in W$. Let $p\in W$. We can write $p = a X^3 + bX^2Y + c XY^2 + d Y^3$. Since $W$ is isotropic, we know that $\omega_{2,\C}(X^3,p) = 0$, hence, by (\ref{eq:omega}), we see that $d=0$. Hence, zero is a common root of every element of $W$. 

Conversely, assume that all elements of $W$ have a common root. By acting with $SL(2,\C)$, we can assume that the common root is zero, hence all elements of $W$ are of the form  $p = a X^3 + bX^2Y + c XY^2$. By (\ref{eq:omega}), 
\[ \omega_{2,\C}(a_1 X^3 + b_1 X^2Y + c_1 XY^2, a_2 X^3 + b_2 X^2Y + c_2 XY^2) = \frac{1}{3}(c_1 b_2 - b_1 c_2)\,. \]
Hence, all polynomials in $W$ have the same ratio $\frac{b}{c}$ (which can be infinite if $c=0$). Given a basis $p_1, p_2$ of $W$, we can multiply one of them by a scalar to make sure they have the same coefficients $b,c$. Then $p_1 - p_2$ is a multiple of $X^3$, hence $W$ has an element with a triple root. 
\end{proof}

We can now prove the following:
\begin{Theorem}
$K_{\C}$ is in bijection with $\C\mathbb{P}^1 \times \C\mathbb{P}^1$.	
\end{Theorem}
\begin{proof}
We construct an explicit bijection 
\[g\co\C\mathbb{P}^1 \times \C\mathbb{P}^1 \stackrel{\cong}{\longrightarrow} K_{\C}\] 
defined by:
\[g(([a:b], [c:d])) = \left\{
	\begin{array}{ll}
		 \xi^2_\C([a:b])  & \mbox{if } [a:b] = [c:d]\\
		\langle (bX-aY)^3, (dX-cY)^2 (bX-aY) \rangle & \mbox{if } [a:b] \neq [c:d]
	\end{array}
\right.\]
The facts that the map $g$ is well-defined and bijective are easy calculations. 
\end{proof}

\begin{Remark}[The space $K_{\R}$]\label{K_R}
The space 
\[K_\R = \bigcup_{t \in \R\mathbb{P}^1} K_{\xi^1_\R(t)} \subset \Lag(\C^{4}) \]
is precisely the space $K_{\rho,I}$ in Section \ref{sub:dod}, hence $\Omega = \Lag(\C^{4}) \setminus K_\R$. Note that 
\[K_{\R} = \{W \in \mathrm{Lag}(\C^{4}) \mid \exists\; p= ( bX- a Y)^3 \in W, [a:b]\in\R\mathbb{P}^1\}\] 
corresponds to the set of Lagrangian subspaces with a triple real root $[a:b] \in \R\mathbb{P}^1$. We thus see that 
\[K_{\R}:= \cup_{t \in \R\mathbb{P}^1} K_{\xi^1_\R(t)}\cong \R\mathbb{P}^1 \times \C\mathbb{P}^1 \cong \R\mathbb{P}^1 \times \mathrm{Lag}(\C^2)\,,\] 
or more precisely $K_{\R} = g(\R\mathbb{P}^1 \times \mathrm{Lag}(\C^2))$.
\end{Remark}

\begin{Remark}[Generalisation to $\mathrm{Lag}(\C^{2n})$]
  The second factor $\C\mathbb{P}^1$ in the maps above should be interpreted as $\mathrm{Lag}(\C^2)$. In fact, in more generality, we can prove that, for any dimension, $K_{\C} \cong \C\mathbb{P}^1 \times \mathrm{Lag}(\C^{2(n-1)})$ and $K_{\R} \cong \R\mathbb{P}^1 \times \mathrm{Lag}(\C^{2(n-1)}).$
\end{Remark}
 
\begin{Lemma}   \label{lem:double_root}
Every Lagrangian $W$ contains a polynomial 
\[p(X,Y) = (b X - a Y)^2 (d X-c Y)\]  
with a double root $[a:b] \in \C\mathbb{P}^1$ and a single root $[c:d]\neq [a:b]\in \C\mathbb{P}^1$. 
\end{Lemma} 
\begin{proof}
If $W$ has a polynomial with a triple root, then $W\in K_\C$, and we saw above that these Lagrangians also contain a polynomial with a double root and a single root. If $W$ does not contain a polynomial with a triple root but has a polynomial with a double root, the third root must be distinct, hence we are done.
 
Assume now that $W$ has a polynomial with three distinct roots. Acting with $SL(2,\C)$, we can assume that the three roots are $[-1:1], [0:1], [1:0]$, hence that $XY(X+Y)=X^2Y + XY^2 \in W$. Let $p\in W$, $p = a X^3 + bX^2Y + c XY^2 + d Y^3$. By (\ref{eq:omega}), we have
\[ \omega_{2,\C}(XY(X+Y), p) = \frac{1}{3}(b-c)\,, \]
which implies that $b=c$.  Hence $W = \left<XY(X+Y), a X^3 + d Y^3\right>$, and $W$ is determined by $[a:d]$. The other elements of $W$ are of the form 
\[q = aX^3 + \beta X^2 Y + \beta X Y^2 + d Y^3\,.\]
We want to find elements with a double root, we can find them using the discriminant. The {\em discriminant} $\Delta$ of a degree--$3$ polynomial $\alpha x^3+\beta x^2y +\gamma xy^2 +\delta y^3$ is defined by 
\[\Delta(\alpha x^3+\beta x^2y +\gamma xy^2 +\delta y^3) := \beta^2\gamma^2-4\alpha\gamma^3-4\delta\beta^3-27\alpha^2\delta^2+18\alpha\beta\gamma\delta\,,\] 
and polynomials with a double root correspond to the zeros of the discriminant. In our case, we have 
\[\Delta(q) = \beta^4 - 4(a+d)\beta^3 + 18ad \beta^2 - 27 a^2d^2\,.  \]
A polynomial has always at least one solution over the complex numbers, hence for every value of $a$ and $d$, we can find a $\beta$ such that $q$ has a double root.
\end{proof}

We can now state the main result for this section, which identifies the unique open $\SL(2, \C)$-orbit in $\mathrm{Lag}(\C^{4})$ with the space of regular ideal tetrahedra in $\mathbb{H}^3$. This will be a key step in the proof of Theorem \ref{fiber_smooth}. 
\begin{Theorem}
There are three $\SL(2, \C)$--orbits in $\mathrm{Lag}(\C^{4})$:
\begin{itemize}
  \item $\xi^2_\C(\C\mathbb{P}^1) = \SL(2, \C) \cdot \langle X^3, X^2 Y \rangle$ is the only closed orbit, and it is in bijection with the diagonal $\Delta \subset K_\C \cong \C\mathbb{P}^1 \times \C\mathbb{P}^1$. 
  \item $K_{\C} \setminus \xi^2_\C(\C\mathbb{P}^1) =  \SL(2, \C) \cdot \langle X^3, X Y^2 \rangle$ is not open nor close, and it is in bijection with $\C\mathbb{P}^1 \times \C\mathbb{P}^1 \setminus \Delta$.
  \item $\mathrm{Lag}(\C^{4}) \setminus K_{\C} =  \SL(2, \C) \cdot \langle X^2Y, X^3 + Y^3\rangle$ is the only open orbit and it is in bijection with the space $\mathfrak{T}_{\HH^3}$ of regular ideal hyperbolic tetrahedra in $\mathbb{H}^3$.
\end{itemize}
\end{Theorem}
 
Recall that an ideal hyperbolic tetrahedron is called {\em regular} when all the dihedral angles are equal (and equal to $\frac{\pi}{3}$). These tetrahedra can also be characterized by their volume or their cross-ratio, since a tetrahedron is regular if and only if it has maximal volume, if and only if the cross-ratio of its vertices is $\frac{1-\sqrt{3}i}{2}.$ Recall that given $4$ points $z_1, z_2, z_3, z_4$ in $\C\mathbb{P}^1$ we define their {\em cross-ratio} as 
\[[z_1, z_2, z_3, z_4] = \frac{(z_3 - z_1) (z_4 - z_2)}{(z_3 - z_2)(z_4 - z_1)}\,.\] 
Equivalently, $[z_1, z_2, z_3, z_4] = A z_4$, where $A \in \PSL(2,\C)$ is defined by $A z_1 = \infty,  Az_2 = 0$, and $A z_3 = 1$. Note that if you change the order of the points, then the cross-ratio $z$ can become $1-\frac{1}{z}$ or $\frac{1}{1-z}$. On the other hand, if $z_0 = \frac{1-\sqrt{3}i}{2}$, then $z_0 = 1-\frac{1}{z_0} = \frac{1}{1-z_0}$, so the characterization of regular tetrahedra doesn't depend on the chosen order of the vertices when calculating the cross-ratio, as it should be. We will discuss more properties of regular ideal hyperbolic tetrahedra in Section \ref{tetrahedra}.


\begin{proof}
	We have already discussed above the bijection $g\co\C\mathbb{P}^1 \times \C\mathbb{P}^1 \stackrel{\cong}{\longrightarrow} K_{\C}$. From the discussion above, you can see that:
	\begin{itemize}
		\item $\xi^2_\C(\C\mathbb{P}^1) = \SL(2, \C) \cdot \langle X^3, X^2 Y \rangle$ corresponds to Lagrangians all of whose polynomials share a common double root;
		\item $K_{\C} \setminus \xi^2_\C(\C\mathbb{P}^1) =  \SL(2, \C) \cdot \langle X^3, X Y^2 \rangle$ corresponds to Lagrangians all of whose polynomials share a common single root.
	\end{itemize}
Since $\xi^2_\C$ is an embedding, $\xi^2_\C(\C\mathbb{P}^1) \cong \C\mathbb{P}^1$ is closed.

To complete the proof, we need to show that $\mathrm{Lag}(\C^{4}) \setminus K_{\C}$ is in bijection with $\mathfrak{T}_{\HH^3}$, and it is one $\SL(2, \C)$ orbit. Given $W \not\in K_{\C}$, by Lemma \ref{lem:double_root}, $W$ contains a polynomial 
with a double root and a single root. Acting with $SL(2,\C)$, we can assume the roots are $[0:1]$ and $[1:0]$, i.e. that $X^2 Y \in W$. Let $p\in W$, $p = a X^3 + bX^2Y + c XY^2 + d Y^3$. By (\ref{eq:omega}) we have
\[ \omega_{2,\C}(X^2Y, p) = -\frac{1}{3}c\,, \]
which implies that $c=0$.  Hence $W = \left<X^2Y, a X^3 + d Y^3\right>$. Acting with $SL(2,\C)$, we can fix $[0:1]$ and $[1:0]$ and send $[a:d]$ to $[1:1]$. Then we have $W= \left<X^2Y, X^3 + Y^3\right>$, and this proves that 
\[\mathrm{Lag}(\C^{4}) \setminus K_{\C} =  \SL(2, \C) \cdot \langle X^2Y, X^3 + Y^3\rangle \,.\] 


Lastly, in order to see that this open orbit is in bijection with $\mathfrak{T}_{\HH^3}$, we study the Lagrangian subspace $W$ and see it contains exactly $4$ `special' polynomials which have a double root. To find them, we use the discriminant as in the proof of Lemma~\ref{lem:double_root}. The elements of $W$ are of the form 
\[q = \alpha X^3 + \beta X^2 Y + \alpha Y^3\,,\]
hence the discriminant is 
\[\Delta(q) = -\alpha (4\beta^3+27\alpha^3)\,.\] 
We have that $\Delta(q)= 0$ if and only if 
\begin{enumerate}
  \item $[\alpha:\beta]=[0:1]$;
  \item $[\alpha:\beta] = \left[-\frac{\sqrt[3]{4}}{3}:1\right]$;
  \item $[\alpha:\beta] = \left[\frac{1}{3\sqrt[3]{2}}- i \frac{1}{\sqrt[3]{2}\sqrt{3}}:1\right]$;
  \item $[\alpha:\beta] = \left[\frac{1}{3\sqrt[3]{2}}+ i \frac{1}{\sqrt[3]{2}\sqrt{3}}:1\right]$.
\end{enumerate}
The associated polynomials have double and single roots, respectively, given by:
\begin{enumerate}
  \item $0$ and $\infty$;
  \item $\sqrt[3]{2}$ and $-\frac{1}{\sqrt[3]{4}}$;
  \item $\frac{-1-i\sqrt{3}}{\sqrt[3]{4}}$ and $\frac{1+i\sqrt{3}}{2\sqrt[3]{4}}$;
  \item $\frac{-1+i\sqrt{3}}{\sqrt[3]{4}}$ and $\frac{1-i\sqrt{3}}{2\sqrt[3]{4}}$.
\end{enumerate}

The vertices corresponding to the $4$ double roots define an ideal hyperbolic tetrahedron $T = \left\{ 0,\sqrt[3]{2},\frac{-1-i\sqrt{3}}{\sqrt[3]{4}}, \frac{-1+i\sqrt{3}}{\sqrt[3]{4}}\right\}$ and the vertices corresponding to the $4$ single roots define a `dual ' ideal hyperbolic tetrahedra $T_{dual} = \left\{\infty,-\frac{1}{\sqrt[3]{4}}, \frac{1+i\sqrt{3}}{2\sqrt[3]{4}}, \frac{1-i\sqrt{3}}{2\sqrt[3]{4}}  \right\}$. The  tetrahedron $T_{dual}$ is the image of $T$ by the central symmetry centered at the barycenter of $T$. A simple calculation shows that the cross-ratio of the vertices of $T$ and $T_{dual}$ are equal to $\frac{1-\sqrt{3}i}{2}$:
\[\left[ 0, \sqrt[3]{2}, \frac{-1-i\sqrt{3}}{\sqrt[3]{4}}, \frac{-1+i\sqrt{3}}{\sqrt[3]{4}} \right] = \left[\infty, -\frac{1}{\sqrt[3]{4}}, \frac{1+i\sqrt{3}}{2\sqrt[3]{4}}, \frac{1-i\sqrt{3}}{2\sqrt[3]{4}} \right] = \frac{1-\sqrt{3}i}{2}\,.\]
Hence $T$ and $T_{dual}$ are regular ideal hyperbolic tetrahedra, as we wanted to prove.
\end{proof}

Now we want to see that the bijections described above are actually homeomorphisms. Let us first clarify or recall the topology of these two spaces. The topology on $\mathfrak{T}_{\HH^3} \cup \C\mathbb{P}^1\times \C\mathbb{P}^1$ is defined by the topology of $\C\mathbb{P}^1$ and the fact that both $\mathfrak{T}_{\HH^3}$ and $\C\mathbb{P}^1\times \C\mathbb{P}^1$ are subspaces of the symmetrization $\mathrm{Sym}^4\left(\C\mathbb{P}^1\right)$. Remember also that the topology of $\mathrm{Lag}(\C^{4})$ can be described by considering $\mathrm{Lag}(\C^{4}) \subset \mathrm{Gr}(2, \C^{4})$ and using the Pl\"ucker coordinates for the topology of $ \mathrm{Gr}(2, \C^{4})$.

\begin{Theorem}\label{thm:lagrange}
The space $\mathrm{Lag}(\C^{4})$ is homeomorphic to the space 
\[\mathfrak{T}_{\overline{\HH^3}} := \mathfrak{T}_{\HH^3} \cup \C\mathbb{P}^1\times \C\mathbb{P}^1\,.\]
\end{Theorem}

\begin{Remark}
With this theorem in mind, we will call elements of $\C\mathbb{P}^1\times \C\mathbb{P}^1$ \textit{degenerate tetrahedra}, and the first coordinate in $\C\mathbb{P}^1\times \C\mathbb{P}^1$ their \textit{degenerate barycenter}.  
\end{Remark}


\begin{proof}
To prove this result, we will extend the map $g$, defined above, to a map 
\[g\co \mathfrak{T}_{\HH^3} \cup \C\mathbb{P}^1\times \C\mathbb{P}^1 \to \mathrm{Lag}(\C^{4})\,,\] 
and check that the map is a homeomorphism. 

Given a tetrahedron $T = \{v_1, \ldots, v_4\} \in \mathfrak{T}_{\HH^3}$, we define the \textit{dual tetrahedron} $T_{dual} = \{v_1^{dual}, \ldots, v_4^{dual}\}$ as the tetrahedron (also in $\mathfrak{T}_{\HH^3}$) with vertices $v_i^{dual}$ such that $v_i$, the barycenter $b$ of $T$ and $v_i^{dual}$ lie on the same geodesic for $i = 1, \ldots, 4$. Let $v_i$ and $v_j$ be two distinct vertices of $T$. If we let $v_i = [a_1: b_1] \in  \C\mathbb{P}^1$, $v_i^{dual} = [c_1: d_1] \in  \C\mathbb{P}^1$, $v_j = [a_2: b_2] \in  \C\mathbb{P}^1$, and $v_j^{dual} = [c_2: d_2] \in  \C\mathbb{P}^1$, then the Lagrangian subspace associated to $T$ is 
\[g(T) = W := \langle (b_1 X - a_1 Y)^2 (d_1 X - c_1 Y), (b_2 X - a_2 Y)^2 (d_2 X - c_2 Y) \rangle \in \mathrm{Lag}(\C^{4})\,.\] Its Pl\"ucker coordinates are:
\begin{itemize}
	\item $W_{1,2} = -b_1^2d_1(b_2^2c_2+2a_2b_2d_2) + b_2^2d_2 (b_1^2c_1+2a_1b_1d_1)$;
	\item $W_{1,3} = b_1^2d_1(a_2^2d_2+2a_2b_2c_2) - b_2^2d_2 (a_1^2d_1+2a_1b_1c_1)$;
	\item $W_{1,4} = -b_1^2d_1a_2^2c_2 + b_2^2d_2a_1^2c_1$;
	\item $W_{2,3} = -(b_1^2c_1+2a_1b_1d_1)(a_2^2d_2+2a_2b_2c_2) + (b_2^2c_2+2a_2b_2d_2)(a_1^2d_1+2a_1b_1c_1)$;
	\item $W_{2,4} = a_2^2c_2 (b_1^2c_1+2a_1b_1d_1) - a_1^2c_1 (b_2^2c_2+2a_2b_2d_2)$;
	\item $W_{3,4} = -a_2^2c_2 (a_1^2d_1+2a_1b_1c_1) + a_1^2c_1(a_2^2d_2+2a_2b_2c_2)$.
\end{itemize}
Similarly, given a point $([a:b], [c:d]) \in \C\mathbb{P}^1\times \C\mathbb{P}^1 \setminus \Delta$, where $\Delta = \{([a:b], [a:b]) \in \C\mathbb{P}^1\times \C\mathbb{P}^1\}$ is the diagonal, then the associated Lagrangian subspace is 
\[U = \langle (b X - a Y)^3, (b X - a Y)(d X - c Y)^2 \rangle \in \mathrm{Lag}(\C^{4})\,,\] 
which has Pl\"ucker coordinates:
\begin{itemize}
	\item $U_{1,2} = 2b^3 d (bc-ad)$;
	\item $U_{1,3} = b^2(3ad+bc)(bc-ad)$;
	\item $U_{1,4} = -ba(bc+ad)(bc-ad)$;
	\item $U_{2,3} = -3ba(bc+ad)(bc-ad)$;
	\item $U_{2,4} = a^2 (ad+3bc)(bc-ad)$;
	\item $U_{3,4} = -2a^3c (bc-ad)$.
\end{itemize}
Lastly, given a point $([a:b], [a:b]) \in \Delta \subset \C\mathbb{P}^1\times \C\mathbb{P}^1$, the associated Lagrangian subspace is 
\[Z = \langle (b X - a Y)^3, (b X - a Y)^2 (d X - c Y) \rangle \in \mathrm{Lag}(\C^{4})\,,\] 
for $[c:d]\neq[a:b]\in \C\mathbb{P}^1$, and $Z$ has Pl\"ucker coordinates:
\begin{itemize}
	\item $Z_{1,2} = - 2b^4 (bc-ad)$;
	\item $Z_{1,3} = 2ab^3(bc-ad)$;
	\item $Z_{1,4} = -a^2b^2(bc-ad)$;
	\item $Z_{2,3} = -3a^2b^2(bc-ad)$;
	\item $Z_{2,4} = 2a^3b (bc-ad)$;
	\item $Z_{3,4} = -a^4 (bc-ad)$.
\end{itemize}
We first notice that the tetrahedra in $\mathfrak{T}_{\HH^3}$, since they have maximal volume, can only degenerate so that the barycenter also degenerates (in the sense that it converges to a point in $\C\mathbb{P}^1$). In that case, at least three vertices will converge to the same point in $\C\mathbb{P}^1$, so tetrahedra can only degenerate to points in $\C\mathbb{P}^1\times \C\mathbb{P}^1$, that is $\C\mathbb{P}^1\times \C\mathbb{P}^1 = \partial \mathfrak{T}_{\HH^3}$. 

We can now see that $g|_{\C\mathbb{P}^1\times \C\mathbb{P}^1}$ and $g|_{\mathfrak{T}_{\HH^3}}$ are continuous. For the first case, we only have to consider the expression of the Pl\"ucker coordinates and look at the case $([a_n:b_n], [c_n:d_n]) \in \C\mathbb{P}^1\times \C\mathbb{P}^1 \setminus \Delta$ such that $([a_n:b_n], [c_n:d_n]) \to ([a:b], [a:b]) \in \Delta$. In particular, we can do the calculations in the case that $[a:b]= [0:1]\in \C\mathbb{P}^1$. If we denote $U_{i, j}^n$ the Pl\"ucker coordinates associated to $F\left(([a_n:b_n], [c_n:d_n])\right)$, we can see that the only non-zero coordinate in the limit is $U_{1,2}^n$, as we wanted. For the second case $g|_{\mathfrak{T}_{\HH^3}}$, again, we only have to consider the expression of the Pl\"ucker coordinates in term of the vertices of the tetrahedra. Hence we are left with the discussion of converging sequences $\{T_n\}$ of tetrahedra in $\mathfrak{T}_{\HH^3}$ such that $T_n \to T_\infty \in \C\mathbb{P}^1\times \C\mathbb{P}^1$. We have two possibilities:
\begin{itemize}
	\item $T_\infty = ([a:b], [c:d])\in \C\mathbb{P}^1\times \C\mathbb{P}^1 \setminus \Delta$.
	\item $T_\infty = ([a:b], [a:b])\in \Delta$.
\end{itemize} 

In the first case, three vertices of the tetrahedron $\{T_n\}$  and all the dual vertices of the tetrahedron $\{T_n^{dual}\}$ converge to $[a:b] \in \C\mathbb{P}^1$, while in the second case all the four vertices of the tetrahedron $\{T_n\}$ and at least three dual vertices of the tetrahedron $\{T_n^{dual}\}$ converge to $[a:b] \in \C\mathbb{P}^1$. In particular, in the first case we choose vertices $v_i^n = [a_1^n: b_1^n] \in  \C\mathbb{P}^1$ and $v_j^n = [a_2^n: b_2^n] \in  \C\mathbb{P}^1$ of $T_n$ with dual vertices $(v_i^n)^{dual} = [c_1^n: d_1^n] \in  \C\mathbb{P}^1$ and $(v_j^n)^{dual} = [c_2^n: d_2^n] \in  \C\mathbb{P}^1$, such that
\begin{itemize}
	\item $[a_1^n:b_1^n]\to [a:b]$;
	\item $[c_1^n:d_1^n]\to [a:b]$;
	\item $[a_2^n:b_2^n]\to [c:d]$;
    \item $[c_2^n:d_2^n]\to [a:b]$.
\end{itemize} 
We can also assume $[a:b]= [0:1]$ and $[c:d]= [1:0]$. If we denote $W_{i, j}^n$ the Pl\"ucker coordinates associated to $g\left(T_n\right)$, we can see that the only non-zero coordinate in the limit is $W_{1,3}^n$, as we wanted.

In the second case we choose vertices $v_i^n = [a_1^n: b_1^n] \in  \C\mathbb{P}^1$ and $v_j^n = [a_2^n: b_2^n] \in  \C\mathbb{P}^1$ of $T_n$ with dual vertices $(v_i^n)^{dual} = [c_1^n: d_1^n] \in  \C\mathbb{P}^1$ and $(v_j^n)^{dual} = [c_2^n: d_2^n] \in  \C\mathbb{P}^1$, such that
\begin{itemize}
	\item $[a_1^n:b_1^n]\to [a:b]$;
	\item $[a_2^n:b_2^n]\to [a:b]$.
\end{itemize} 
Again, we can assume $[a:b]= [0:1]$. Let's denote $W_{i, j}^n$ the Pl\"ucker coordinates associated to $F\left(T_n\right)$. We have two cases:
\begin{itemize}
	\item At least one of the sequences $[c_1^n:d_1^n]$ or $[c_2^n:d_2^n]$ do not converge to $[a:b]$, then we can see that the only non-zero coordinate in the limit is $W_{1,2}^n$, as we wanted. 
	\item If both $[c_1^n:d_1^n], [c_2^n:d_2^n]$ converge to $[a:b]$, then we need to be more careful, and analyze the rate of convergence, but after diving all coordinates by $c_1^n$ or $c_2^n$, we can see that the only non-zero coordinate in the limit is $W_{1,2}^n$, as we wanted. 
\end{itemize} 
\end{proof}


The geometric picture discussed in the last step of the proof above using tetrahedra and degenerate tetrahedra and their barycenters inspired the definition of the continuous projection below. Let 
\[\pi_{\beta}\co \mathfrak{T}_{\overline{\HH^3}} \to \overline{\HH^3} := \HH^3 \cup \C\mathbb{P}^1\] 
be the map defined by sending each (possibly degenerate) tetrahedron to its (possibly degenerate) barycenter. By considering the map $Q:= \pi_{\beta} \circ g^{-1}$ we obtain:
\begin{Corollary}
	There is a continuous $\SL(2,\C)$--equivariant projection 
	\[Q\co\mathrm{Lag}(\C^{4}) \to \overline{\HH^3}\,.\]
\end{Corollary}


\subsection{The domain $\Omega \subset \mathrm{Lag}(\C^{4})$}

We define the map
\[\pi_{\mathcal{P}}\co\overline{\HH^3}\to\overline{\HH^2}\]
as the orthogonal projection into the hyperbolic plane $\mathcal{P}$ bounded by $\R\mathbb{P}^1\subset \C\mathbb{P}^1$. This map is only $\SL(2,\R)$--equivariant. By composition, we obtain a projection
\[\pi_{\mathcal{P}} \circ Q \co \mathrm{Lag}(\C^{4}) \to \overline{\HH^2}\,.\]
By Remark \ref{K_R}, the inverse image of $\R\mathbb{P}^1 = \partial \HH^2$ is the set $K_\R$, and the inverse image of $\HH^2$ is the set
\[\Omega = \mathrm{Lag}(\C^{4}) \setminus K_{\R}\,.\] 
Restricting the map $\pi_{\mathcal{P}} \circ Q$ to $\Omega$, we obtain  
\[q = \pi_{\mathcal{P}} \circ Q|_{\Omega}\co \Omega \to \HH^2\,,\]
an $\SL(2,\R)$--equivariant map from $\Omega$ to $\HH^2$, which is a fiber bundle by Lemma~\ref{lem:fiber bundle}.


We will identify $\HH^2$ with the hyperbolic plane $\mathcal{P} \subset \HH^3$. 
We denote by $\mathcal{O} \in \HH^2 \subset \HH^3$ the point $\mathcal{O} = (0, 1)\in \C \times \R_{>0}$, and by $F_q$ the fiber of $q$ over this point:
\[F = q^{-1}(\mathcal{O}) \subset  \Omega\,.\]
Since $q\co \Omega \to \HH^2$ is a locally trivial fibration over a contractible base, we conclude the following result: 
\begin{Corollary}\label{productF}
  The space $\Omega$ is homeomorphic to the product $F \times \HH^2$, hence $\Omega$ deformation retracts to $F$.
\end{Corollary}

In the following sections we will describe the topology of $F$. Since $F$ is homotopy equivalent to our smooth fibre $\mathfrak{F}$, this information will allow us to determine $\mathfrak{F}$. 

\section{Spaces of regular tetradera} \label{sub:new_york}

We consider the geodesic $\overline{\ell}:= \pi_\mathcal{P}^{-1}(\mathcal{O})$, where  $\pi_\mathcal{P}\co\overline{\HH^3}\to\overline{\HH^2}$ is the orthogonal projection. The geodesic $\overline{\ell}$ joins the points at infinity $i$ and $-i$. We denote $\ell := \overline{\ell} \cap \HH^3$, so we have $\overline{\ell} = \ell \cup \{\pm i\}$. We denote by $\ell^+$ the ray of $\ell$ from $\mathcal{O}$ to $i$, and by $\ell^-$ the ray from $\mathcal{O}$ to $-i$. In both cases, $\mathcal{O}$ is included, and $\pm i$ is not. Similarly, we denote by $\overline{\ell^+}$ and $\overline{\ell^-}$ the compactified rays that include $\pm i$.  

We will identify $\overline{\ell}$ with the segment $[-\infty, \infty]$ via the homeomorphism 
\[\eta \co \overline{\ell} \to [-\infty, \infty]\]
defined by the following properties 
\begin{itemize}
  \item $\eta (\mathcal{O}) = 0$,
  \item $\eta (\pm i) = \pm \infty$,
  \item for every $x \in \ell^+$, $\eta(x) = d_{\HH^3}(\mathcal{O}, x)$, and
  \item for any $x \in \ell^-$, $\eta(x) = -d_{\HH^3}(\mathcal{O}, x)$.
\end{itemize}  

We define the space $\mathfrak{T}_{\overline{\ell}}$ consisting of (possibly degenerate) tetrahedra with (possibly degenerate) barycenter on the geodesic $\overline{\ell}$. This space is homeomorphic to the fiber $F$: recall that in Theorem \ref{thm:lagrange} we constructed an explicit homeomorphism from the space of (unlabelled) regular ideal tetrahedra to the Lagrangian Grassmannian
\[g\co \mathfrak{T}_{\overline{\HH^3}} =\mathfrak{T}_{\HH^3} \cup \C\mathbb{P}^1\times \C\mathbb{P}^1 \to \mathrm{Lag}(\C^{4})\,.\] 
Then, the fiber $F =  q^{-1}(\mathcal{O}) \subset  \Omega\subset \mathrm{Lag}(\C^{4})$ for the projection 
$q = \pi_{\mathcal{P}} \circ Q|_{\Omega\co \Omega \to \HH^2}$ is exactly the image $g(\mathfrak{T}_{\overline{\ell}})$.  

Our main aim in this section will be to describe the space $\mathfrak{T}_{\overline{\ell}}$. It will be useful to distinguish between three subsets: the open subset $\mathfrak{T}_\ell$ consisting of tetrahedra with barycenter in $\ell$, and the closed subsets $\mathfrak{T}_i$ and $\mathfrak{T}_{-i}$ consisting of degenerate tetrahedra with barycenter in $i$ or $-i$ respectively. 

We have that
\[\mathfrak{T}_i = \{i\}\times \C\mathbb{P}^1, \hspace{2cm} \mathfrak{T}_{-i}=\{-i\}\times  \C\mathbb{P}^1\,.\]

The space $\mathfrak{T}_\ell$ will be described in Section \ref{tetrahedra}. The shape of each of the three pieces, $\mathfrak{T}_\ell, \mathfrak{T}_i, \mathfrak{T}_{-i}$ is easy to understand. The most interesting thing is to describe how they are glued together, which is done in Section \ref{nyconstruction}. 

We consider the upper half space model $\HH^3 = \C \times \R_{>0}$ of hyperbolic space. In this model, the compactified hyperbolic space is $\overline{\HH^3} = \C \times \R_{\geq 0} \cup \{\infty\}$ and its boundary is $\partial\HH^3 = \C \times \{0\} \cup \{\infty\} = \C\mathbb{P}^1$. In the following, with a slight abuse of notation, we will simply use complex numbers (or $\infty$) to denote points of $\partial\HH^3 = \mathbb{CP}^1$. 
We identify $\HH^2$ with the plane $\mathcal{P} = \R \times \R_{>0} \subset \HH^3$ whose boundary is $\partial\HH^2 = \R\mathbb{P}^1 \subset \C\mathbb{P}^1$. Note that $\PSL(2, \R)$ acts preserving $\mathcal{P}$.

%

\subsection{Regular Ideal Hyperbolic Tetrahedra}\label{tetrahedra}

For any $c\in \HH^3$, let $\mathfrak{T}_c$ be the set of regular ideal unlabelled tetrahedra with barycenter $c$. All the spaces $\mathfrak{T}_c$ are homeomorphic to each other.  The space $\mathfrak{T}_\ell$ is homeomorphic to $\mathfrak{T}_c \times \R$ for any choice of $c$. We will now describe  $\mathfrak{T}_c$. 

Note that for all $c\in \ell$ the space $\mathfrak{T}_c$ is homeomorphic to 
\[\mathfrak{T}_c \cong (\mathbb{T}^{1}(\mathbb{S}^2))/A_4 \cong \mathrm{SO}(3)/A_4 \cong \mathbb{T}^{1, orb}(\mathbb{S}^2/A_4)\,.\]
This is a Seifert fibered space --- an orbifold--$\bS^1$--bundle over the $2$--orbifold $\bS^2(2, 3, 3) = \bS^2 / A_4$ --- and it corresponds to the space described by Martelli \cite{mar_ani} in the second line of Table 10.6 for $q = -2$. The structure of Seifert fibered manifold of $\mathfrak{T}_c$ can be described geometrically. 
 
Consider the action of $\SO(2)$ on $\HH^3$ via rotations that fix $\ell$. Since we are assuming that $c\in \ell$, this induces an action of $\SO(2)$ on $\mathfrak{T}_c$. The orbits of this action are the fibers of the Seifert fibration. The three circles associated with the three singular fibers correspond to tetrahedra (with barycenter at the point $c$) with special symmeries: 
\begin{enumerate}[(i)]
  \item the circles associated with the order--$3$ cone points correspond to tetrahedra in $\mathfrak{T}_c$ with one vertex in $-i$ or $i$, respectively, 
  \item the circle associated with the order--$2$ cone point corresponds to tetrahedra in $\mathfrak{T}_c$ with two sides orthogonal to $\ell$. 
\end{enumerate}

We want to decompose $\mathfrak{T}_c$ in two sets:
\[\mathfrak{T}_c = \mathfrak{T}_{c}^{\uparrow} \cup \mathfrak{T}_{c}^{\downarrow}, 
\] 
where  $c \in \ell$.

As above, we denote by $\ell^+_c$ the ray of $\ell$ from $c$ to $i$, and by $\ell^-_c$ the ray from $c$ to $-i$. In both cases, $c$ is included, and $\pm i$ are not. Similarly, we denote by $\overline{\ell^+_c}$ and $\overline{\ell^-_c}$ the compactified rays that include $\pm i$. The boundary of the tetrahedra in the family in $(ii)$ (with order--$2$ symmetry) will intersect $\ell$ in two points: one in $\ell^-_c$ (which we will denote $A_{c}$) and the other in $\ell^+_c$. Let $B_{c}$ be the point in $\ell^-_c$ between $-i$ and $A_c$ and at (hyperbolic) distance $1$ from $A_{c}$. Let $C_{c}$ be circle in $\C\mathbb{P}^1$ that bounds the plane in $\HH^3$ orthogonal to $\ell$ and intersecting it at $B_c$ and let $D_{c}$ (resp. $\overline{D_{c}}$) be the open (resp. closed) disk in $\C\mathbb{P}^1$ with boundary $C_{c}$ and containing $-i$.  

With this we can write 
\[\mathfrak{T}_c = \mathfrak{T}_{c}^{\uparrow} \cup \mathfrak{T}_{c}^{\downarrow},
\] 
where
\begin{itemize}
  \item $\mathfrak{T}_{c}^{\uparrow}$ is the set of tetrahedra in $\mathfrak{T}_c$ such that all vertices are in $\C\mathbb{P}^1 \setminus D_{c}$;
  \item $\mathfrak{T}_{c}^{\downarrow}$ is the set of tetrahedra in $\mathfrak{T}_c$ such that one of their vertices is in $D_{c}$.
\end{itemize}
The set $\mathfrak{T}_{c}^{\uparrow}$ is closed, and the set $\mathfrak{T}_{c}^{\downarrow}$ is open. They share a common boundary
$\partial \mathfrak{T}_{c}^{\uparrow} = \partial \mathfrak{T}_{c}^{\downarrow}$, the set of tetrahedra in $\mathfrak{T}_c$ such that one of their vertices is in $C_{c}$. The following is true:
\begin{Proposition}\label{unique}
  For every $c \in \ell$ and for every $T \in \overline{\mathfrak{T}_{c}^{\downarrow}}$ there is exactly one vertex in $\overline{D_{c}}$.
\end{Proposition}
 
 \begin{proof}
First we notice that the (hyperbolic) distance between the barycenter and the faces of a regular tetrahedron $T \in \mathfrak{T}_{\HH^3}$ is $\ln{\sqrt{2}}$. To do that, we consider the tetrahedron $T$ with vertices $\{\infty, -1, \frac{1+\sqrt{3}i}{2}, \frac{1-\sqrt{3}i}{2}\}$. We can check that $T$ is regular by calculating the cross ratio $[\infty, -1, \frac{1+\sqrt{3}i}{2}, \frac{1-\sqrt{3}i}{2}]$. The barycenter of this tetrahedron is $(0, 0, \sqrt{2})$, which corresponds to the point of intersection between the geodesic passing through $0$ and $\infty$ and the geodesic passing through $-1$ and orthogonal to the plane $\{x = \frac{1}{2}\}$ in $\HH^3$ (which is the plane containing $\infty$, $\frac{1+\sqrt{3}i}{2}$ and $\frac{1-\sqrt{3}i}{2}$).  It is easy now to se that the (hyperbolic) distance between the barycenter and the face of $T$ passing trough $\{-1, \frac{1+\sqrt{3}i}{2}, \frac{1-\sqrt{3}i}{2}\}$ is $\ln{\sqrt{2}}$. we can then also calculate the dual tetrahedron $T^{dual} = \{0, \frac{1}{2}, \frac{1-\sqrt{3}i}{4}, \frac{1+\sqrt{3}i}{4}\}$.

Second, we notice that, given a tetrahedron with with order--$2$ symmetry, the (hyperbolic) distance between its barycenter $c$ and the point $A_c$ is $\ln\left(\tfrac{1}{2}(\sqrt{6}-\sqrt{2})\right)$. To prove this, we consider the tetrahedron $T = \{1, -1, (2-\sqrt{3})i, -(2-\sqrt{3})i\}$. The barycenter of this tetrahedron is 
$(0, 0, \tfrac{1}{2}(\sqrt{6}-\sqrt{2}))$,
and the distance between the barycenter and the geodesic between $1$ and $-1$ (or equivalently, the (hyperbolic) distance between the barycenter and the geodesic between $(2-\sqrt{3})i$ and $-(2-\sqrt{3})i$) is $\ln\left(\tfrac{1}{2}(\sqrt{6}-\sqrt{2})\right)$.

Finally, given the calculations above, we can conclude the proof. Let $c \in \ell$, and let $T \in \partial \mathfrak{T}_{c}^{\downarrow}$ be a tetrahedra with one vertex $v_T$ on the circle $C_{c}$. Then the other three vertices of $T$ (different from $v_T$) lies on a circle $C_T$, spanning a disc $D_T \in \HH^3$ perpendicular to the geodesic between $v_T$ and $c$ and intersecting it at the point at distance $\mathrm{ln}\sqrt{2}$ from $c$ and farthest from $v_T$. Then, in order to prove the result above, we just have to check that check that the circle $C_T$ does not intersect $\overline{D}_{c}$. 
    \end{proof}
  

We can describe more precisely the topology of $\mathfrak{T}_{c}^{\uparrow}$ and $\mathfrak{T}_{c}^{\downarrow}$.

\begin{Proposition}\   \label{prop:upanddown}
\begin{itemize}
  \item $\mathfrak{T}_{c}^{\uparrow}$ is homeomorphic to the complement of an open tubular neighborhood of a $(2, 3)$--torus knot (or, equivalently, a trefoil knot);
  \item $\mathfrak{T}_{c}^{\downarrow}$ is homeomorphic to a solid torus.
\end{itemize}
 \end{Proposition}

\begin{proof}
Since we know that $\mathfrak{T}_{c} \cong \mathrm{SO}(3)/A_4$, its Seifert structure is well known, see, for example, the second line of Table 10.6 in Martelli \cite{mar_ani} with $q = -2$. 
In order to prove the first claim, we need to understand the Seifert structure of the trefoil knot complement. This is described in Moser \cite{Moser}. We can easily see that it is the same as the one of $\mathfrak{T}_{c}^{\uparrow}$.

For the second claim, we will describe explicit coordinates for $\mathfrak{T}_{c}^{\downarrow}$ as we will need them in the following section. We consider the family $\mathfrak{T}_{c}^{\downarrow, 3}$ of the tetrahedra in $\mathfrak{T}_c^{\downarrow}$ with bottom vertex in $-i$. This is one of the families of tetrahedra with the order--$3$ symmetry. We introduce a certain parametrization of $\mathfrak{T}_{c}^{\downarrow}\setminus \mathfrak{T}_{c}^{\downarrow, 3} \cong \mathbb{A}\times \bS^1$, where $\mathbb{A} \cong \bS^1\times(0,1)$ as follows. If $T \in \overline{\mathfrak{T}_c^{\downarrow}}\setminus \mathfrak{T}_{c}^{\downarrow, 3}$, let $v_T$ be the unique vertex of $T$ in $\overline{D_{c}}\setminus \{-i\}$. We will parametrize $\overline{D_{c}}\setminus \{-i\}$ with polar coordinates centered at $-i$: 
\[\overline{D_{c}}\setminus \{-i\} \cong \mathbb{A} \cong \bS^1\times (0,1)\,.\] 
We let $\theta_T \in \bS^1 = \R/(\frac{2\pi}{3}\Z)$ be the angle defined by the other three vertices as follows. Once $v_T$ is fixed (and $c$ is fixed), the other three vertices lie in (the boundary of) a totally geodesic plane $P_{T}$ and have an order--$3$ invariance. We need to define what is $0\in \bS^1 \cong \R/(\frac{2\pi}{3}\Z)$ in order to being able to measure the angle $\theta_T$. Consider the totally geodesic plane $Q_T$ orthogonal to $\ell$ and passing through $v_T$; it intersect $P_T$ in two points. Define $0\in \bS^1 = \R/(\frac{2\pi}{3}\Z)$ to be the point with lower height.	
\end{proof}


\subsection{Description of the construction}  \label{nyconstruction}

%




In this section, we are going to study the topology of the space $\mathfrak{T}_{\overline{\ell}}$, and prove that it is homeomorphic to a certain quotient $\mathfrak{T}_{\mathcal{O}} \times [-\infty, +\infty] / \sim$, where $\mathcal{O} = \ell \cap \mathcal{P} \in \ell$. 

In order to define the construction, we need to use the following maps:
 \begin{itemize}
  \item $\iota\co \overline{\HH^3} \to \overline{\HH^3}$ is the reflection in the plane $\mathcal{P}$ with boundary $\R\mathbb{P}^1 \subset \C\mathbb{P}^1$. 
  \item $L_{\lambda}^+\co \overline{\HH^3} \to \overline{\HH^3}$ is the hyperbolic isometry of $\HH^3$ with axis $\ell$ and translation length $\lambda\in \R_{>0}$ and attracting fixed point $i$;
  \item $L_{\lambda}^-\co \overline{\HH^3} \to \overline{\HH^3}$ is the hyperbolic isometry of $\HH^3$ with axis $\ell$ and translation length $\lambda\in \R_{>0}$ and attracting fixed point $-i$.
\end{itemize} 

The isometry $L_{\lambda}^+$ can be used to move a tetrahedron in $\mathfrak{T}$. 

\begin{Proposition}
  The transformations $L_{\lambda}^{\pm}$ satisfy the following properties:
  \begin{enumerate}
    \item $L_{\lambda}^{\pm}(\mathfrak{T}_c) = \mathfrak{T}_{L_{\lambda}^{\pm}(c)}$;
    \item $L_{\lambda}^{\pm}(B_{c}) = B_{L_{\lambda}^{\pm}(c)}$.
  \end{enumerate}
\end{Proposition}

Together with the $L_\lambda^+$, we will also need a companion map that  we will denote by $M_\lambda$. This will be, for every $c\in \ell$, the map
\[M_\lambda: \overline{\mathfrak{T}_c^{\downarrow}} \rightarrow \mathfrak{T}_{L_{\lambda}^+(c)}^{\downarrow}\, \]
that moves the barycenter of the tetrahedra along $\ell$ according to $L_\lambda^+$, but does not move the bottom vertex in $\overline{D_c}$. In order to define $M_{\lambda}$, we use the coordinates on $\mathfrak{T}_c^{\downarrow}$ described in the proof of Proposition \ref{prop:upanddown}. The map $M_\lambda$ is defined as follows: 

\begin{itemize}
  \item If $T\in \mathfrak{T}_{c}^{\downarrow,3}$, let $M_{\lambda}(T) = L_{\lambda}^+(T)$.
  \item For every tetrahedron $T \in \overline{\mathfrak{T}_{c}^{\downarrow}} \setminus \mathfrak{T}_{c}^{\downarrow,3}$, let $M_{\lambda}(T)$ be the tetrahedron in $\mathfrak{T}_{L_{\lambda}^+(c)}^{\downarrow}$ with barycenter in $L_{\lambda}^+(c)$, same `bottom' vertex $v_T$ and same angle $\theta_T$.
\end{itemize}

We now have to prove the following fact:
\begin{Lemma}
$M_{\lambda}$ is continuous on $\overline{\mathfrak{T}_{c}^{\downarrow}}$.
\end{Lemma}

\begin{proof}
The continuity comes from the fact that the definition of the planes $Q_T$ and $P_T$ depends continuously on $T$ and so does the definition of the angle $\theta_T$. More precisely, let's consider a sequence $T_n\in \overline{\mathfrak{T}_c^{\downarrow}}\setminus \mathfrak{T}_{c}^{\downarrow, 3}$ such that $T_n\to T\in \mathfrak{T}_{c}^{\downarrow, 3}$. 
Remember that we can parametrize any $T_n \in \overline{\mathfrak{T}_c^{\downarrow}}\setminus \mathfrak{T}_{c}^{\downarrow, 3}$ with a pair $(v_n, \theta_n)$, where $v_n = v_{T_n}\in \overline{D_{c}}\setminus \{-i\}$. We consider polar coordinates on $\overline{D_{c}}\setminus \{-i\}$, in this way $v_n = (r_n, \phi_n) \in (0, 1] \times \R/(2\pi\Z)$. We can always choose the polar coordinates in such a way that the line $(r,0)$ contains one of the vertices of $T$. The fact that $T_n\to T\in \mathfrak{T}_{c}^{\downarrow, 3}$ implies that $r_n \to 0$. Moreover, it implies that $\theta_n + \phi_n \to [\pi] \in \R/(\frac{2\pi}{3}\Z)$. The map $M_{\lambda}$ leaves the angles $\theta_n$ fixed, hence the sequence $M_\lambda(T_n)$ still converges to $M_\lambda(T)$. This shows the continuity of $M_{\lambda}$, as we wanted.
\end{proof}

For every point $z \in \C\mathbb{P}^1 \setminus \{\pm i\}$, consider the unique hyperbolic plane perpendicular to $\ell$ and containing $z$ in its boundary, and denote by $d$ the intersection of this plane with $\ell$. Denote by $h_z := \eta(d)\in \R$ the \emph{height} of $z$. If $z \in \{\pm i\}$, we define $h_i := +\infty$ and $h_{-i} := -\infty$.

For every tetrahedron in $F$, we denote by $b_T$ its barycenter. For every $c \in \ell$ and for every tetrahedron in $\overline{\mathfrak{T}_c^{\downarrow}}$, we denote by $v_T$ the unique vertex of $T$ in $\overline{D_{c}}$ and by by $h_T$ the height $h_{v_T}$.


\begin{Theorem}\label{phi}
There is a continuous surjective map 
\[\Phi \co \mathfrak{T}_{\mathcal{O}} \times [-\infty, +\infty] \to \mathfrak{T}_{\overline{\ell}}\] 
such that
\begin{enumerate}
    \item For all $T\in \mathfrak{T}_{\mathcal{O}}$, and $s\in [-\infty,+\infty]$, $\Phi(T,-s) = \iota(\Phi(\iota(T),s))$;
    \item For all $T \in \mathfrak{T}_{\mathcal{O}}$, $\Phi(T,0) = T$;
    \item For all $T \in \mathfrak{T}_{\mathcal{O}}$, and $s \in [0,+\infty)$, $\Phi(T,s) \in \{T \in \mathfrak{T} \mid b_T \in \ell^+\}$;
    \item The restriction
    \[\Phi|_{\mathfrak{T}_{\mathcal{O}} \times (-\infty, +\infty)}\co \mathfrak{T}_{\mathcal{O}} \times (-\infty, +\infty) \to \mathfrak{T}_\ell\] 
    is a homeomorphism;
    \item $\Phi(\mathfrak{T}_{\mathcal{O}}^{\uparrow}\times \{+\infty\}) = \{(+ i, + i)\} \in \mathfrak{T}_i$;
    \item The restriction $\Phi_{+\infty}:=\Phi|_{\mathfrak{T}_{\mathcal{O}}^{\downarrow} \times \{+\infty\}} \co \mathfrak{T}_{\mathcal{O}}^{\downarrow} \times \{+\infty\} \to \mathfrak{T}_i \setminus \{(+ i, + i)\}$ is surjective; 
    \item Consider the function $f$ defined by
    \[f\co(-\infty, \eta(B_{\mathcal{O}}))\to (-\infty, +\infty)\] 
    \[f(v) = v + \frac{1}{\eta(B_{\mathcal{O}}) - v} = \frac{v^2 - \eta(B_{\mathcal{O}}) v -1}{v - \eta(B_{\mathcal{O}})}   \,.\]
    Then $f$ is a strictly increasing homeomorphism. For every $z \in \C\mathbb{P}^1$ with $h_z < \eta(B_{\mathcal{O}})$, the $f$-\emph{uplift} of $z$ is the point $z^f := L^+_{\lambda}(z)$, where $\lambda = f(h_z)-h_z$. When $z = -i$, $z^f := -i$. In this way, $h_{z^f} = f(h_z)$.   
    \item When $z \in \C\mathbb{P}^1 \setminus \{i\}$, the fiber of $\Phi$ at the point $(i,z) \in \mathfrak{T}_i$ is the circle
       \[\Phi^{-1}(i, z) = \{\ (T,+\infty) \mid T \in \mathfrak{T}_{\mathcal{O}}^{\downarrow}, \ \  (v_T)^f = z \ \}\,.\]
       consisting of all the tetrahedra with a fixed vertex $v_T \in D_{c}$.
\end{enumerate}
\end{Theorem}
\begin{proof}
For the proof we first construct 
\[\Phi^+ = \Phi|_{\mathfrak{T}_{\mathcal{O}} \times [0,  \infty)} \co (\mathfrak{T}_{\mathcal{O}} \times [0,  \infty)) \to \{T \in \mathfrak{T} \mid b_T \in \overline{\ell^+}\}\] 
with the property that for all $T \in \mathfrak{T}_{\mathcal{O}}$, $\Phi^+(T,0) = T$. Then, we will define
\[\Phi^- = \Phi|_{\mathfrak{T}_{\mathcal{O}} \times (-\infty, 0]} \co (\mathfrak{T}_{\mathcal{O}} \times (-\infty, 0]) \to \{T \in \mathfrak{T} \mid b_T \in \overline{\ell^-}\}\]
by the formula
\[\Phi^-(T,-s) = \iota(\Phi^+(\iota(T),s))\,.\]
The map $\Phi$ will be obtained by glueing $\Phi^+$ and $\Phi^-$.

\begin{figure}[htb]
\centering
\includegraphics[height=7cm]{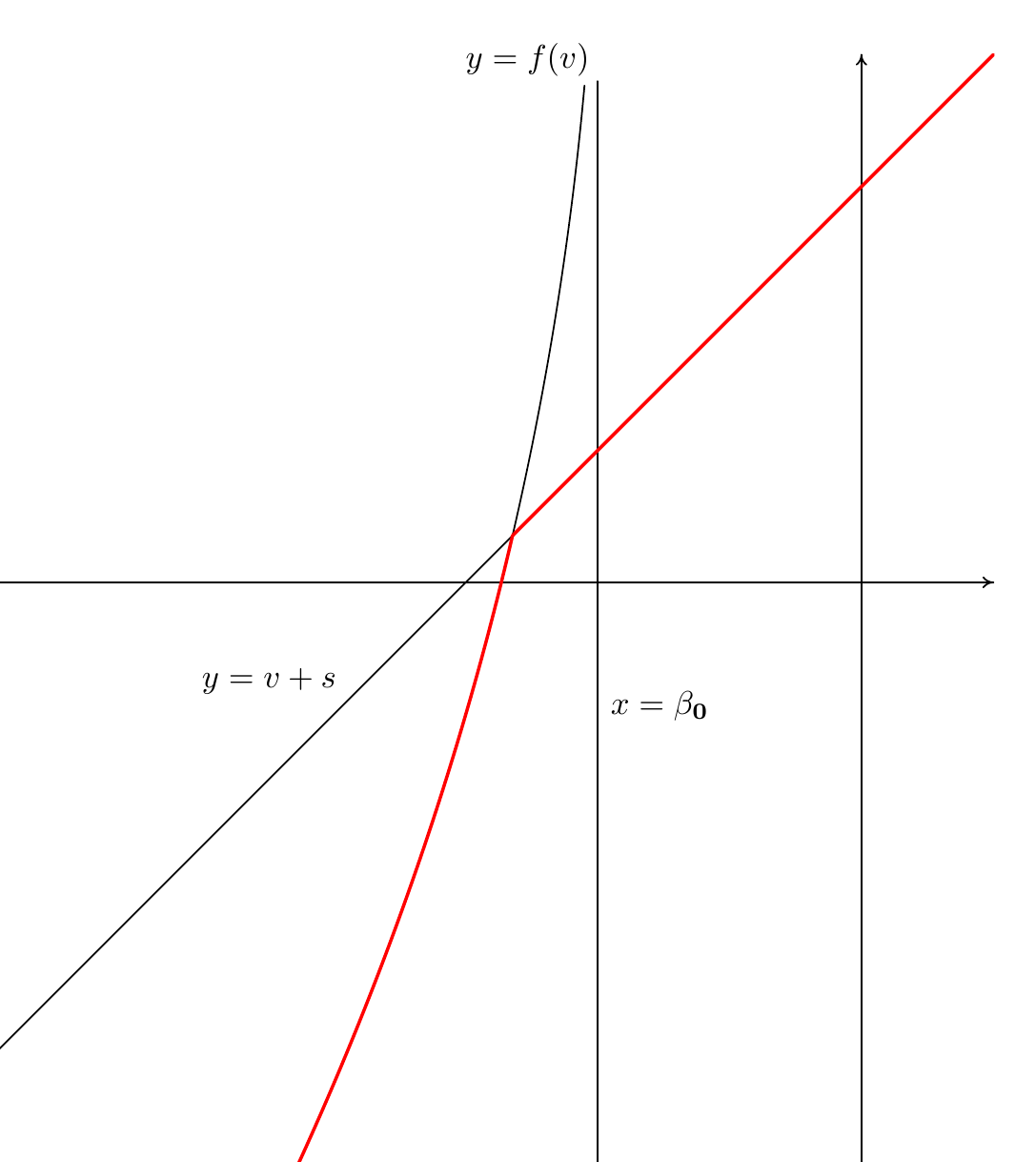}
\caption{The functions $y = v+s$, $y = f(v)$ and in red $y = \mathfrak{T}_s(v) = \mathrm{min}\{v+s, f(v)\}$ (in red).}
\label{f_s}
\end{figure}
  
We will now discuss the construction of $\Phi^+$. 
First of all we notice that property (7) is an easy computation. Moreover, $f$ satisfies the properties 
\begin{enumerate}
\item[(a)] $f(v) > v$.
\item[(b)] $\widehat{f}(v) := f(v)-v = \frac{1}{\eta(B_\mathcal{O})-v}$ is strictly increasing and tends to $+\infty$ when $v \to \eta(B_\mathcal{O})$.
\end{enumerate}   

We define $\Phi^+$ as follows:
\begin{itemize}
  \item If $(T, \infty)\in \mathfrak{T}_{\mathcal{O}}^{\uparrow} \times \{+\infty\}$, then $\Phi^+((T, \infty)) := (+i, +i)$.
  \item If $(T, \infty)\in \mathfrak{T}_{\mathcal{O}}^{\downarrow}\times \{+\infty\}$, then $\Phi^+((T, \infty)) := L^+_{f(h_T) - h_T}(v_T)$.
  \item If $(T, s)\in \mathfrak{T}_{\mathcal{O}}^{\uparrow}\times [0, +\infty)$, then $\Phi^+((T, s)) := L^+_{s}(T)$.
  \item If $(T, s)\in \mathfrak{T}_{\mathcal{O}}^{\downarrow} \times [0, +\infty)$, then $$\Phi^+((T, s)) := \left\{
      \begin{array}{ll}
        L^+_{s}(T)  & \mbox{if } s \leq f(h_T) - h_T \\
        M_{s - f(h_T) + h_T} \circ L^+_{f(h_T) - h_T}(T) & \mbox{if } s \geq f(h_T) - h_T,
      \end{array}
    \right.$$
\end{itemize} 
where the map $M_{\lambda}$ is the one defined before the theorem. In order to better understand the map $\Phi^+$, consider, for every $s \in [0, \infty]$, the map represented in Figure \ref{f_s}, defined by 
\[\rho_s\co(-\infty, \eta(B_\mathcal{O})\to (-\infty, \eta(B_\mathcal{O}) + s)\] 
  \[\rho_s(v) = \mathrm{min}\{v + s, f(v)\}= 
  \begin{cases} v + s &\text{ if } v+s \leq f(v) \\
    f(v) &\text{ if } v+s \geq f(v). 
   \end{cases}\,.\]
Note that for any $T \in \mathfrak{T}_{\mathcal{O}}^{\downarrow}$ and $s \in [0, +\infty)$, $\Phi^+((T, s))$ is a tetrahedron $T'$ with height $h_{T'} = \rho_s(h_T)$. 

We can now check that for any $t \in \mathfrak{T}_{\mathcal{O}}$ we have that $\Phi^+ (T, 0) = L^+_{0}(T) = T = \Phi^- (T, 0)$, so we can combine the maps $\Phi^+$ and $\Phi^-$ into the map $\Phi$ we wanted. This calculation also shows that $\Phi$ satisfies properties $(1)$ and $(2)$. From the definition, we can also see that $\Phi^+$ satisfies property $(5)$. 

\begin{Claim}  \label{claim:homeo}
The map $\Phi$ satisfies property (4), i.e. the restriction
    \[\Phi|_{\mathfrak{T}_{\mathcal{O}} \times (-\infty, +\infty)}\co \mathfrak{T}_{\mathcal{O}} \times (-\infty, +\infty) \to \mathfrak{T}_\ell\] 
    is a homeomorphism.
\end{Claim}
\begin{proof}
For the surjectivity, let $T \in \mathfrak{T}_{c_s}$, where $c_s \in \ell^+$. Let $s = \eta(c_s)\in [0, +\infty)$. 
\begin{itemize}
  \item If $T \in \mathfrak{T}_{c_s}^{\uparrow}$, let $\widehat{T}=L^{-}_s(T)\in \mathfrak{T}_{\mathcal{O}}^{\uparrow}$. Then $\Phi^+(\widehat{T}, s) = T$.
  \item If $T \in \mathfrak{T}_{c_s}^{\downarrow}\setminus \mathfrak{T}_{c_s}^{\downarrow,3}$, let $h = \rho_s^{-1}(h_T)$. We need to use the parametrization of $\mathfrak{T}_{c_s}^{\downarrow}$ described in the definition of the map $M_{\lambda}$. Let $\widehat{T}$ be the (unique) tetrahedra in $\mathfrak{T}_{\mathcal{O}}$ such that $h_{\widehat{T}}= h$ and such that $\theta_{\widehat{T}} = \theta_{T}$. Then again we can see that $\Phi^+(\widehat{T}, s) = T$.
  \item If $T \in \mathfrak{T}_{c_s}^{\downarrow,3}$, let $\widehat{T}=L^{-}_s(T)\in \mathfrak{T}_{\mathcal{O}}^{\uparrow}$. Then $\Phi^+(\widehat{T}, s) = T$.
\end{itemize} 
  
For the injectivity, let $T, T' \in \mathfrak{T}_{\mathcal{O}}$ and $s, s' \in [0, +\infty)$ such that $\Phi^+(T, s) = \Phi^+(T', s') = \widehat{T}$. Since the image is the same, the two image tetrahedra have the same barycenter, so $s = s'$. Also, since the map $\Phi^+$ does not change the ``type'' of the tetrahedra (that is $\mathfrak{T}_{c}^{\uparrow}$ or $\mathfrak{T}_{c_s}^{\downarrow}$), we have two cases: either $T, T' \in \mathfrak{T}_{\mathcal{O}}^{\uparrow}$ or $T, T' \in \mathfrak{T}_{\mathcal{O}}^{\downarrow}$. 
\begin{itemize}
  \item In the first case, the fact that $L_s^{+}$ is an isometry implies that $T = T'$.
  \item In the second case, we have that $h_T = h_{T'} = \rho_s^{-1}(h_{\widehat{T}})$. Now we have two possibilities: either $s \leq f(h_T) - h_T$ or $s \geq f(h_T) - h_T$. 
  \begin{enumerate}
    \item If $s \leq f(h_T) - h_T$, then we use again the fact that $L_s^{+}$ is injective to see that $T = T'$.
    \item If $s \geq f(h_T) - h_T$, then we use again the fact that the map $M_{s - f(h_T) + h_T} L_{f(h_T) - h_T}^{+}$ is injective to conclude that $T = T'$.
  \end{enumerate}
\end{itemize} 

Since all the maps we use are continuous, we only need to check the continuity at points in $\mathfrak{T}_{\mathcal{O}}^{\uparrow} \cap \overline{\mathfrak{T}_{\mathcal{O}}^{\downarrow}}$, that is at tetrahedra $T$ such that $h_T = \eta(B_\mathcal{O})$. Since $f(\eta(B_\mathcal{O})) = + \infty$, then we are always in the case $s \leq f(h_T) - h_T$, so $\Phi^+(T, s) = L^+_{s}(T)$.
\end{proof}

\begin{Claim} The map $\Phi$ satisfies property (6), i.e. 
the restriction $\Phi_{+\infty}:=\Phi|_{\mathfrak{T}_{\mathcal{O}}^{\downarrow} \times \{+\infty\}} \co \mathfrak{T}_{\mathcal{O}}^{\downarrow} \times \{+\infty\} \to \mathfrak{T}_i \setminus \{(+i,+i)\}$ is surjective.
\end{Claim}
This also shows, together with Claim \ref{claim:homeo}, that $\Phi$ is surjective.
\begin{proof}
Given a point $(i,z)\in \{+ i\} \times (\C\mathbb{P}^1\setminus \{+ i\})$, since the function $f$ is a homeomorphism, we can find a point $v \in \C\mathbb{P}^1$ such that $f(h_v) = h_z$ and such that $v^f = z$. Let $T \in \mathfrak{T}_{\mathcal{O}}$ be such that $v_T = v$. Then $T$ is necessarily in $\mathfrak{T}_{\mathcal{O}}^{\downarrow}$. By definition of the map $\Phi^+$, $\Phi^+(T) = (i,z)$.    
\end{proof}

Now, we only have to prove: 
\begin{Claim}
The map $\Phi^+$ is continuous.
\end{Claim}

\begin{proof}
The continuity on $\mathfrak{T}_{\mathcal{O}} \times [0,+\infty)$ was established in Claim \ref{claim:homeo}. The continuity on $\mathfrak{T}_{\mathcal{O}} \times \{+\infty\}$ is clear from the definition. So, in order to check the continuity of $\Phi^+$ it suffices to check the continuity for sequences $(T_n, s_n) \in \mathfrak{T}_{\mathcal{O}} \times [0, +\infty)$ such that $T_n \to T$ and $s_n \to +\infty$. We have two cases:
  \begin{enumerate}[(1)]
    \item If $T$ is in the interior part of $\mathfrak{T}_{\mathcal{O}}^{\uparrow}$, then we can assume that all the $T_n$ are also in $\mathfrak{T}_{\mathcal{O}}^{\uparrow}$. From the definition of the map, $$\Phi^+(T_n,s_n) = L^+_{s_n}(T_n)  \to (+i,+i) = \Phi^+(T,+\infty),$$ because all the $T_n$ are in $\mathfrak{T}_{\mathcal{O}}^{\uparrow}$. 
    \item If $T\in \mathfrak{T}_{\mathcal{O}}^{\downarrow}$, then we can assume that all the $T_n$ are also in $\mathfrak{T}_{\mathcal{O}}^{\downarrow}$.
     Since $T_n \to T$, when $n$ is big enough, we can assume that $h_{T_n}$ is close enough to $h_T$. 
     Hence, when $n$ is big enough, we can assume that $s_n \geq f(h_{T_n}) - h_{T_n}$.  From the definition, we have  $$\Phi^+(T_n, s_n) = M_{s_n - f(h_{T_n}) + h_{T_n}} \circ L^+_{f(h_{T_n}) - h_{T_n}}(T_n).$$ 
     This is a tetrahedron $T_n'$ with vertex $v_{T_n'}$ equal to $L^+_{f(h_{T_n}) - h_{T_n}}(v_{T_n})$. Hence, the sequence $\Phi^+(T_n, s_n)$ converges to $\Phi^+(T, +\infty) = L^+_{f(h_{T}) - h_{T}}(v_{T})$.   
    \item Finally, we assume that $T \in \partial \mathfrak{T}_{\mathcal{O}}^{\uparrow}$. If a subsequence of $T_n$ lies in $\mathfrak{T}_{\mathcal{O}}^{\uparrow}$, we can conclude that subsequence converges to $(i,i)$ as in part (1). Now let's assume that all the $T_n$s are in $\mathfrak{T}_{\mathcal{O}}^{\downarrow}$, and write  $T_n' = \Phi^+(T_n, s_n)$. Then, $h_{T_n} \to +\infty$ and also $f(h_{T_n}) - h_{T_n} \to +\infty$, hence, for big enough $n$, we can assume that $s_n$ is as big as we want, and $f(h_{T_n}) - h_{T_n}$ is as big as we want. From this we see that $h_{T_n'}$ becomes as big as we want, hence $T_n' \to (+i,+i)$.
  \end{enumerate}
\end{proof}
This concludes the proof of the theorem.
\end{proof}

On $\mathfrak{T}_{\mathcal{O}} \times [-\infty, +\infty]$, we consider the following equivalence relation: for $T,T' \in \mathfrak{T}_{\mathcal{O}}$ and $t,t' \in [-\infty, +\infty]$
\[ (T,t)  \sim (T',t')  \Leftrightarrow     
\begin{cases}
T,T' \in \mathfrak{T}_{\mathcal{O}}^{\uparrow}, t=t'=+\infty, \text{ or} \\
T,T' \in \iota(\mathfrak{T}_{\mathcal{O}}^{\uparrow}), t=t'=-\infty, \text{ or} \\
T,T' \in \mathfrak{T}_{\mathcal{O}}^{\downarrow}, v_T = v_{T'}, t=t'=+\infty, \text{ or} \\
T,T' \in \iota(\mathfrak{T}_{\mathcal{O}}^{\downarrow}), v_{\iota(T)} = v_{\iota(T')}, t=t'=-\infty, \text{ or} \\
T=T', t=t'
\end{cases}
\]

\begin{Corollary}
The map $\Phi$ from Theorem \ref{phi} descends to a homeomorphism
\[ \bar{\Phi}: \mathfrak{T}_{\mathcal{O}} \times [-\infty, +\infty] / \sim \ \longrightarrow \ \mathfrak{T}_{\overline{\ell}} \,.\]
\end{Corollary}
\begin{proof}
The description of the fibers of the map $\Phi$ given in Theorem \ref{phi} guarantees that the map descends to the quotient. It is continuous and onto because $\Phi$ is continuous and onto. It is 1-1, again from the description of the fibers. It is a homeomorphism because it is a bijective continuous map from a compact space to a Hausdorff space. This concludes the proof.
\end{proof}

\section{Topology of the fiber} \label{fiber}

In this section we continue the study of the fiber $F= q^{-1}(\mathcal{O})$ of the projection $q \co \Omega \to \HH^2$, where $\mathcal{O} = (0, 1)\in \C \times \R_{>0}$. In the end we will use the study of the topology of $F$ to describe the homeomorphism type of the smooth fiber $\mathfrak{F}$.
We start by analyzing the structure of $\mathfrak{T}_{\overline{\ell}}$ in a bit more detail.

\subsection{Singularities of the fiber $F$} \label{sub:fiber}
The space $F \cong \mathfrak{T}_{\overline{\ell}}$ is not a manifold. We will show in this subsection that it has four singular points, and all the other points have neighborhoods homeomorphic to $\R^4$. The four singular points are $(+i,+i), (+i,-i) \in \mathfrak{T}_i$ and $(-i,-i), (-i,+i) \in \mathfrak{T}_{-i}$. Two of them, $(+i,-i)$ and $(-i,+i)$, are `mild singularities' -- they are orbifold points with isotropy group $\mathbb{Z}_3$. The other two singular points, $(+i,+i)$ and $(-i,-i)$, are more complicated singularities, and a small neighborhood of these points looks like the cone over a closed $3$--manifold that is a Dehn filling of the trefoil knot. All such Dehn fillings are described in Moser \cite{Moser}. As a consequence, we will prove Corollary~\ref{cor:non-smooth fibration}, stating that the fibration $q$ is not a smooth map.  
%
%

We already know, by part (4) of Theorem~\ref{phi}, that $\mathfrak{T}_{\ell}$ is a manifold. We will now describe a neighborhood of the points of $\mathfrak{T}_i$ and $\mathfrak{T}_{-i}$. We only need to discuss $\mathfrak{T}_i$, because we have the orientation reversing homeomorphism $\iota$ that exchanges $\mathfrak{T}_{-i}$ with $\mathfrak{T}_i$.  

We first describe the `mild' singular points and the manifold points.

\begin{Proposition}
Every point $\mathfrak{T}_i$, except from the two points $(+i,+i)$ and $(+i,-i)$, has a neighborhood in $\mathfrak{T}_{\overline{\ell}}$ that is homeomorphic to $\R^4$. The point $(+i,-i)$ has a neighborhood in $\mathfrak{T}_{\overline{\ell}}$ that is homeomorphic to $\R^4$ modded out by a linear action of $\mathbb{Z}_3$.
\end{Proposition}
\begin{proof}
A \emph{labelled tetrahedron} is a tuple $(T,v_1,v_2,v_3,v_4)$, where $T$ is a tetrahedron and $\{v_1,v_2,v_3,v_4\}$ is the set of vertices of $T$. We say that the labelling is \emph{even} if, when watching from the vertex $v_1$, the vertices $v_2,v_3,v_4$ appear in counter-clockwise cyclic order. The labelling is \emph{odd} otherwise. An even labelled tetrahedron is determined by its baricenter $b_T$ and the first two vertices $v_1, v_2$. The vertices $v_3$ and $v_4$ are determined by these data. 

We denote by $\mathfrak{T}_\mathcal{O}^{even}$ the set of all even labelled tetrahedra with barycenter in $\mathcal{O}$.  The group $A_4$ acts on $\mathfrak{T}_\mathcal{O}^{even}$ in the following way: if $\sigma \in A_4$, define
\[\sigma \cdot (T,v_1,v_2,v_3,v_4) = (T,v_{\sigma(1)},v_{\sigma(2)},v_{\sigma(3)},v_{\sigma(4)})\,.\]
We have a natural forgetful map
\[r : \mathfrak{T}_\mathcal{O}^{even} \ni (T,v_1,v_2,v_3,v_4) \ \longrightarrow \ T \in \mathfrak{T}_\mathcal{O} \]
that is $12 : 1$. This map identifies $\mathfrak{T}_\mathcal{O}$ with a quotient:
\[ \mathfrak{T}_\mathcal{O} = \mathfrak{T}_\mathcal{O}^{even}/A_4\,.\]

The space $\mathfrak{T}_\mathcal{O}^{even}$ is homeomorphic to $SO(3) \simeq \R\mathbb{P}^3 \simeq T^1(\mathbb{S}^2)$. To explicitly see the homeomorphism between   $\mathfrak{T}_\mathcal{O}^{even}$ and $T^1(\mathbb{S}^2)$, notice that we can see $v_1$ as a point of $\mathbb{S}^2$, identify the circle where the other three vertices lie with the tangent circle at $v_1$, and then see $v_2$ as a unit tangent vector to the point $v_1$. An interesting consequence is that $\mathfrak{T}_\mathcal{O} \simeq SO(3)/A_4$. 

We define the open subset $\mathfrak{T}_\mathcal{O}^{even, \downarrow}$ as
\[ \mathfrak{T}_\mathcal{O}^{even, \downarrow} = \{\ (T,v_1,v_2,v_3,v_4) \in \mathfrak{T}_\mathcal{O}^{even}  \ \mid \ v_1 \in D_{\mathcal{O}}\,\}. \]
The subset $\mathfrak{T}_\mathcal{O}^{even, \downarrow}$ is not preserved by the action of $A_4$. Only the subgroup of $A_4$ that fixes $1$ acts there. This subgroup is isomorphic to $\mathbb{Z}_3$. 

The restriction of $r$ to $\mathfrak{T}_\mathcal{O}^{even, \downarrow}$ gives a $3:1$ map
\[ r| : \mathfrak{T}_\mathcal{O}^{even, \downarrow} \ \longrightarrow \ \mathfrak{T}_\mathcal{O}^{\downarrow}\]
that identifies $\mathfrak{T}_\mathcal{O}^{\downarrow}$ with a quotient
\[\mathfrak{T}_\mathcal{O}^{\downarrow} = \mathfrak{T}_\mathcal{O}^{even, \downarrow}/\mathbb{Z}_3\,. \]

We now apply a construction called mapping cylinder to $\mathfrak{T}_\mathcal{O}^{even}$. The {\em mapping cylinder} of $p\co \mathbb{T}^1(\mathbb{S}^2) \to \mathbb{S}^2$ is the space 
\[M_p = T^{\leq 1}(\mathbb{S}^2) = \left((T^1(\mathbb{S}^2) \times [0,1]) \sqcup \mathbb{S}^2\right)/ \sim\,,\] 
where $\sim$ is defined by $(y, 0) \sim p(y)$. Note that $M_p$ corresponds to the unit disk bundle of $\bS^2$, and its boundary is the unit tangent bundle $\partial M_p \cong T^1(\mathbb{S}^2) \cong \R\mathbb{P}^3 \cong \mathrm{SO}(3)$. In particular, $M_p$ is a manifold with boundary.  

We now take its double, i.e. we glue two copies of $M_p$ along their boundary via the identity map: 
\[ M := M_p\sqcup_{id_\partial} M_p \]
The space $M$ is clearly a manifold, and it is not hard to see that it is indeed homeomorphic to the manifold $\bS^2 \times \bS^2$, even if we will not need this fact here. 


Now, it is also clear from the definition that we have a map
\[ \Psi:\mathfrak{T}_\mathcal{O}^{even} \times [-\infty,\infty]\ \longrightarrow\ M\,. \]
This map identifies $M$ with the quotient  
\[ M \simeq \mathfrak{T}_\mathcal{O}^{even} \times [-\infty,\infty] / \sim\,,\]
where the equivalence relation $\sim$ identifies all the labelled tetrahedra in $\mathfrak{T}_\mathcal{O}^{even} \times \{\infty\}$ that have the same vertex $v_1$, and, similarly, identifies all the labelled tetrahedra in $\mathfrak{T}_\mathcal{O}^{even} \times \{-\infty\}$ that have the same vertex $v_1$.  

Now, let's consider the following open subset of $\mathfrak{T}_\mathcal{O}^{even} \times [-\infty,\infty]$:
\[U = {\mathfrak{T}_\mathcal{O}^{even, \downarrow}} \times [-\infty,\infty] \subset \mathfrak{T}_\mathcal{O}^{even} \times [-\infty,\infty] \]
The image $\Psi(U)$ is an open subset of $M$, hence it is a manifold. 

Now let's consider again the $3:1$ map $r_|: {\mathfrak{T}_\mathcal{O}^{even, \downarrow}} \rightarrow \mathfrak{T}^{\downarrow}_{\mathcal{O}}$. This induces the map 
\[\Phi \circ (r_| \times \mathrm{Id}) : {\mathfrak{T}_\mathcal{O}^{even, \downarrow}} \times [-\infty, \infty] \rightarrow F\,,\]
where $\Phi$ is the map from Theorem \ref{phi}. The image of this map is an open subset $V$ of $F$ that contains $\mathfrak{T}^{deg^+} \setminus \{(+i,+i)\}$. Moreover, $V$ is homeomorphic with the quotient $\Psi(U)$ by the action of the group $\mathbb{Z}_3$. This shows that all the points of $\mathfrak{T}^{deg^+} \setminus \{(+i,+i),(+i,-i)\}$ are manifold points in $F$, and that the point $(+i,-i)$ is an orbifold point with group $\Z_3$.  
%
%
\end{proof}
       
We now describe a neigborhood of the singular point $(+i,+i)$.
\begin{Proposition}       
The point $(+i,+i)$ has a neighborhood in $\mathfrak{T}$ that is homeomorphic to the cone $C(M)$ over a closed $3$-manifold $M$, where $M$ is a Dehn filling of the complement of the trefoil knot.     
\end{Proposition}   

Note that all the possible Dehn fillings of the trefoil knot are described in \cite{Moser}.

\begin{proof}
Using the notation of Section \ref{tetrahedra}, Let $\beta_\mathcal{O}' = \beta_\mathcal{O}-1$, and consider the disc $D_{\beta_\mathcal{O}'}$, an open disc in $\C\mathbb{P}^1$ contained in $D_{\beta_\mathcal{O}}$. 
We define the closed subset $\mathfrak{T}_\mathcal{O}^*$ of $\mathfrak{T}_\mathcal{O}$ as the set of tetrahedra in $\mathfrak{T}_\mathcal{O}$ such that all vertices are in $\C\mathbb{P}^1 \setminus D_{\beta_{c}'}$. This is a closed neighborhood of $\mathfrak{T}_\mathcal{O}^\uparrow$, and it is homeomorphic to $\mathfrak{T}_\mathcal{O}^\uparrow$. 

Now consider the set 
\[ U := \Phi(\mathfrak{T}_\mathcal{O}^* \times [0,+\infty]) \subset \mathfrak{T}\,,\]
where $\Phi$ is the map from Theorem \ref{phi}. The set $U$ is a closed neighborhood of $(+i,+i)$ in $F$, and it is easy to see that $U$ is a cone with center in $(+i,+i)$ over the boundary $\partial U$. We only need to prove that $\partial U$ is homeomorphic to a Dehn filling of the trefoil knot complement. 

The boundary $\partial U$ is the union of two pieces, $\Phi(\mathfrak{T}_\mathcal{O}^* \times \{0\})$ and $\Phi(\partial \mathfrak{T}_\mathcal{O}^* \times [0,+\infty])$. The first piece, $\Phi(\mathfrak{T}_\mathcal{O}^* \times \{0\})$ is homeomorphic to  $\mathfrak{T}_\mathcal{O}^*$ i.e. homeomorphic to $\mathfrak{T}_\mathcal{O}^\uparrow$, and by Proposition \ref{prop:upanddown} this is homeomorphic to the trefoil knot complement. The second piece is homeomorphic to a solid torus: indeed $\partial \mathfrak{T}_\mathcal{O}^*$ is a torus, 
$\Phi(\partial \mathfrak{T}_\mathcal{O} \times [0,+\infty))$ is homeomorphic to a torus times $[0,+\infty)$, and $\Phi(\partial \mathfrak{T}_\mathcal{O} \times \{+\infty\})$ is a circle that completes the solid torus. 

From this we can see that $\partial U$ is a Dehn filling of the trefoil knot complement. 
\end{proof}

\begin{Corollary}  \label{cor:non-smooth fibration}
The fiber bundle $q\co\Omega\to\HH^2$ is not smooth.
\end{Corollary}
\begin{proof}
The map $q$ is $SL(2,\R)$-equivariant. If it were smooth, it would have some regular values, and by $SL(2,\R)$-equivariance all the values would be regular. Hence, it would be a submersion, and this would imply that the fiber would be a smooth manifold, which is impossible because it has four singular points.  
\end{proof}

\subsection{Cohomology of the fiber $F$} \label{sub:homology}

In this section we will study the cohomology of $\mathfrak{T}_{\overline{\ell}} \cong F$, and this will determine the cohomology for $\mathfrak{F}$. In particular, we will prove:
\begin{Proposition}\label{hom}
	$F$ is a a Poincar\'e duality space, it is simply connected and its homology is given by:
     \begin{itemize}
      \item $H^0(F;\Z) \cong \Z$;
      \item $H^1(F;\Z) = H^3(F;\Z) = 0$;
      \item $H^2(F;\Z) \cong \Z \oplus \Z$;
      \item $H^4(F; \Z) \cong \Z$;
      \item $H^i(F;\Z) =0$ for all $i > 4$,
    \end{itemize}
	Moreover, for each $i= 0, \ldots, 4$ there is a natural isomorphism $H^i(F; \Z) \cong H_{4-i}(F; \Z)$.
\end{Proposition}


%

To prove Proposition~\ref{hom}, we  work again with $\mathfrak{T}_{\overline{\ell}}$. We will write $\mathfrak{T}_{\overline{\ell}}$ as a union of two open sets
\[ A := \{ x \in \mathfrak{T}_{\overline{\ell}} \ \mid\ \eta(b_x) \in (-1,+\infty] \}, \hspace{1cm} B := \{ x \in \mathfrak{T}_{\overline{\ell}} \ \mid\ \eta(b_x) \in [-\infty,+1) \}\,.\]
Let \[ Y := A \cap B \cong \mathfrak{T}_{\mathcal{O}} \times (-1, 1). \]
The first thing we have to prove is the following:

\begin{Proposition}\label{retract}
The open set $A$ deformation retracts to 
\[  \{ x \in A \ \mid\ \eta(b_x)= +\infty\} = \mathfrak{T}_i \cong \C\mathbb{P}^1 \,. \]
Similarly, $B$ deformation retracts to 
\[ \{ x \in B \ \mid\ \eta(b_x)= -\infty\}   = \mathfrak{T}_i \cong \C\mathbb{P}^1  \,. \]
\end{Proposition}

\begin{proof}
Denote by $\kappa = (\kappa_1,\kappa_2)$ the inverse of the map
\[\Phi|_{\mathfrak{T}_{\mathcal{O}} \times (-\infty, +\infty)}\co \mathfrak{T}_{\mathcal{O}} \times (-\infty, +\infty) \to \mathfrak{T}_\ell \]
which is an homeomorphism by Theorem \ref{phi}. So for $x \in A \setminus \mathfrak{T}_i$, $\kappa_1(x) \in \mathfrak{T}_{\mathcal{O}}$, $\kappa_2(x) \in (-1,+\infty)$, and $\Phi(\kappa_1(x),\kappa_2(x)) = x$. Here, $\kappa_2(x) = \eta(b_x)$.
We write the retraction as
\[H : A \times [0,+\infty] \rightarrow A \]
\[ H(x,t) := 
\begin{cases}
x & \text{ if } x\in \mathfrak{T}_i\\
\Phi(\kappa_1(x),\kappa_2(x)+t)  & \text{ if } x\in A \setminus \mathfrak{T}_i
\end{cases}
\]
It is easy to check that $H$ is a retraction by deformation, i.e. $H$ is continuous, $H(\cdot,0)$ is the identity on $A$, $H(x,+\infty) \in \mathfrak{T}_i$ and for all $x \in \mathfrak{T}_i, H(x,t) = x$. 
\end{proof}

We will also use the following version of Poincar\'e Duality to calculate the homology of the intersection $Y$. Note that we use the convention that homology and cohomology groups of negative dimension are zero, so the duality statement includes the fact that all the non-trivial homology and cohomology of $M$ lies in the dimension range from $0$ to $n$.

\begin{Theorem}[Poincar\'e Duality, see Hatcher {\cite[page 231]{hat_alg}}]\label{poincare}
  Let $M$ be a closed orientable $n$--manifold. Then:
  \begin{enumerate}
    \item $H_k(M;\Z)$ and $H^{n-k}(M;\Z)$ are isomorphic.
	\item Modulo their torsion subgroups, $H_k(M;\Z)$ and $H_{n-k}(M;\Z)$ are isomorphic.
    \item The torsion subgroups of $H_{k}(M;\Z)$ and $H_{n-k-1}(M;\Z)$ are isomorphic for $k = 0, \ldots, 4$.
  \end{enumerate}
\end{Theorem}

\begin{proof}[Proof of Proposition \ref{hom}]
	The fact that $F$ is a Poincar\'e duality space follows from the fact that $F$ is homotopically equivalent to the smooth manifold $\mathfrak{F}$, since both  are homotopically equivalent to $\Omega$, as proven in Corollary \ref{productF} and Theorem \ref{thm:TopologyQuotient}. Since $F$ and $\mathfrak{T}_{\overline{\ell}}$ are homeomorphic, this also tells us that $\mathfrak{T}_{\overline{\ell}}$ is a Poincar\'e duality space. 
	
	For the second part of the result, we will study the homology and cohomology of $\mathfrak{T}_{\overline{\ell}}$, which will suffice to conclude. Through all the proof, we will use the decomposition above for $\mathfrak{T}_{\overline{\ell}} = A \cup B$ and $Y = A\cap B$. Using Proposition \ref{retract} and the definitions we can see the following homotopy equivalences:
	    \begin{itemize}
	      \item $A, B \simeq \bS^2$, and
	      \item $Y \simeq \mathrm{SO}(3)/A_4$.
	    \end{itemize}
	
	The simple connectivity of $\mathfrak{T}_{\overline{\ell}}$ follows from Seifert-Van Kampen theorem and the decomposition $\mathfrak{T}_{\overline{\ell}} = A \cup B$ described above.  Proposition \ref{retract} shows that the groups $\pi_1(A)$ and $\pi_1(B)$ are trivial, and hence that $\pi_1(\mathfrak{T}_{\overline{\ell}})$ is trivial as well. Since $\mathfrak{T}_{\overline{\ell}}$ is connected, this proves that $\mathfrak{T}_{\overline{\ell}}$ is simply connected. 
	

  
In order to compute the cohomology of $\mathfrak{T}_{\overline{\ell}}$, we will use Mayer-Vietoris sequence. We know the cohomology of $A$ and $B$:
  \begin{itemize}
    \item $H^0(A;\Z) \cong H^0(B;\Z) \cong \Z$;
	\item $H^2(A;\Z) \cong H^2(B;\Z) \cong \Z$;
    \item $H^i(A;\Z) = H^i(B;\Z) =0$ for $i = 1$ and for all $i > 2$.
  \end{itemize}  
In order to compute the cohomology of $Y$, we remember that it deformation retracts to $\mathrm{SO}(3)/A_4$, which is a Seifert Fiber manifold described in the second line of Table 10.6 in Martelli \cite{mar_ani} with $q = -2$. Hence the cohomology for $Y$ is:
    \begin{itemize}
    \item $H^0(Y;\Z)\cong \Z$;
    \item $H^1(Y;\Z) = 0$;
    \item $H^2(Y;\Z) \cong \Z_3$;
	\item $H^3(Y;\Z) \cong \Z$;
    \item $H^i(Y;\Z) =0$ for all $i > 3$.
  \end{itemize}
 Since $Y$ is connected we have $H_0(Y;\Z) \cong H^3(Y;\Z)\cong \Z$, and since it is a manifold, and hence a Poincar\'e duality space, we have that $H_3(Y;\Z) \cong H^0(Y;\Z) \cong \Z$. From Martelli \cite{mar_ani} we can see that $H_1(Y;\Z) \cong H^2(Y;\Z) \cong \Z_3$. Finally, again using Poincar\'e duality (see Theorem \ref{poincare} with $n = 3$), we can see that $H_2(Y;\Z) \cong H^1(Y;\Z)$ is free (because its torsion subgroup is isomorphic to the one of $H_0(Y;\Z)$), and that $H_2(Y;\Z) = H^1(Y;\Z) = 0$ (because $H_1(Y;\Z)\cong \Z_3$ and, modulo their torsion subgroups, $H_1(M;\Z)$ and $H_{2}(M;\Z)$ are isomorphic). 
 
 Now we are ready to compute the cohomology of $\mathfrak{T}_{\overline{\ell}}$. We have:
  \begin{itemize}
   \item $H^0(\mathfrak{T}_{\overline{\ell}};\Z)\cong \Z$;%
   \item $H^1(\mathfrak{T}_{\overline{\ell}};\Z) = H^3(\mathfrak{T}_{\overline{\ell}};\Z ) = 0$;
   \item $H^2(\mathfrak{T}_{\overline{\ell}};\Z) \cong \Z \oplus \Z$;
   \item $H^4(\mathfrak{T}_{\overline{\ell}};\Z) \cong \Z$;%
   \item $H^i(\mathfrak{T}_{\overline{\ell}};\Z) =0$ for all $i > 4$.
 \end{itemize}
Since $\mathfrak{T}_{\overline{\ell}}$ is connected, we have $H_0(\mathfrak{T}_{\overline{\ell}};\Z)\cong H^4(\mathfrak{T}_{\overline{\ell}};\Z) \cong \Z$. Using Theorem \ref{poincare}, we can see that $H_4(\mathfrak{T}_{\overline{\ell}};\Z)\cong H_0(\mathfrak{T}_{\overline{\ell}};\Z)\cong \Z$, that $H_1(\mathfrak{T}_{\overline{\ell}};\Z) = H_3(\mathfrak{T}_{\overline{\ell}};\Z) = H^1(\mathfrak{T}_{\overline{\ell}};\Z) = H^3(\mathfrak{T}_{\overline{\ell}};\Z) = 0$, and that $H_2(\mathfrak{T}_{\overline{\ell}};\Z) \cong H^2(\mathfrak{T}_{\overline{\ell}};\Z)$ is free abelian. We now use the following exact sequence coming from Mayer--Vietoris sequence to see that $H^2(\mathfrak{T}_{\overline{\ell}};\Z)\cong \Z \oplus \Z$. 
 \begin{equation}\label{eq2}
   0 = H^1(Y;\Z) \to H^2(\mathfrak{T}_{\overline{\ell}};\Z) \to H^2(A;\Z) \oplus H^2(B;\Z) \to H^2(Y;\Z) \to H^3(\mathfrak{T}_{\overline{\ell}};\Z) = 0.
 \end{equation}
 As we said, using the fact that $F$ is homeomorphic to $\mathfrak{T}_{\overline{\ell}}$ via the map $g^{-1}$, the result follows.
\end{proof}

\subsection{The homeomorphism type of $\mathfrak{F}$} \label{sub:smooth}

%
%
In this section we will prove the following result: 
\begin{Proposition}\label{fib}
  $\mathfrak{F}$ is homeomorphic to $\C\mathbb{P}^2 \# \overline{\C\mathbb{P}}^2$, and $F$ and $\mathfrak{T}_{\overline{\ell}}$ are homotopically equivalent to $\C\mathbb{P}^2 \# \overline{\C\mathbb{P}}^2$.
\end{Proposition}


For the proof we need deep classification theorems of simply connected smooth $4$--manifolds due to Whitehead, Milnor, Milnor--Hausemoller, Freedman, Serre and Donaldson, which use their intersection form.  The \textit{intersection form} for a closed oriented $4$--manifold $N$ is the map
\[Q_N \co H^2(N; \Z) \times H^2(N; \Z) \to H^4(N; \Z) \to \Z\]
defined by $Q_N(\alpha, \beta) : =  (\alpha \smile \beta) [N],$ where $\alpha, \beta \in H^2(N, \Z)$ and  $\smile$ denotes the cohomological cup product of $\alpha$ and $\beta$ and $[N]\in H_4(N;\Z)$ is the fundamental class. 
See Scorpan \cite[Chap. 3]{sco_the} for a more detailed discussion. 

This definition of the intersection form only uses the cup product, and this is well defined for all topological spaces, including our singular space $\mathfrak{T}_{\overline{\ell}}$. In our proof, we want to compute the intersection form of the smooth $4$-manifold $\mathfrak{F}$, and we will do this by computing the cup product of the homotopically equivalent space $\mathfrak{T}_{\overline{\ell}}$. 

Recall the following definitions:
\begin{itemize}
	\item The intersection form $Q_N$ is called \textit{unimodular} if the matrix representing $Q_N$ is invertible over $\Z$.
	\item The \textit{rank} of $Q_N$ is defined as $\mathrm{rank}(Q_N):= \mathrm{dim}_\Z H^2(N; \Z)$.
	\item The \textit{signature} of $Q_N$ as $$\mathrm{sign}(Q_N): = \mathrm{dim}_\Z H^2_+(N; \Z) - \mathrm{dim}_\Z H^2_-(N; \Z),$$ where $H^2_+(N; \Z)$ (resp. $H^2_-(N; \Z)$) is defined as the maximal positive-definite (resp. negative-definite) subspace for $Q_N$.
	\item The \textit{definiteness} of $Q_N$ can be positive definite, negative definite or indefinite. We say that $Q_N$ is \textit{positive-definite} if for all non-zero $\alpha$, we have $Q_N(\alpha, \alpha)> 0$, and \textit{negative-definite} if for all non-zero $\alpha$, we have $Q_N(\alpha, \alpha) < 0$. If there exists classes $\alpha$, and $\beta$ such that $Q_N(\alpha, \alpha)> 0$ and $Q_N(\beta, \beta)< 0$, then $Q_N$ is called \textit{indefinite}.
	\item The \textit{parity} of $Q_N$ can be even or odd. We say that $Q_N$ is {\em even} if for all classes $\alpha$ we have $Q_N(\alpha, \alpha)$ is even. Otherwise we say that $Q_N$ is {\em odd}. 
\end{itemize}

As proven in Scorpan \cite[Sec. 3.2]{sco_the}, the intersection form $Q_N$ of a $4$--manifold is always unimodular. In addition, the intersection form satisfies the following properties: given two $4$--manifolds $N_1$ and $N_2$, we have
\begin{itemize}
	\item $Q_{\overline{N}} = -Q_N,$ where $\overline{N}$ is $N$ with opposite orientation.
	\item $Q_{N_1 \# N_2} = Q_{N_1} \oplus Q_{N_2}$, where $N_1 \# N_2$ is the connected sum of $N_1$ and $N_2$.
\end{itemize} 

\begin{Example}
	In Scorpan \cite[Sec. 3.2]{sco_the} one can see the details of the calculations of the intersection forms for  $\C\mathbb{P}^2 \# \overline{\C\mathbb{P}}^2$ and $\mathbb{S}^2 \times \mathbb{S}^2$ which are simply-connected $4$--manifolds with their intersection forms of rank $2$ and indefinite signature. The parity is odd for $\C\mathbb{P}^2 \# \overline{\C\mathbb{P}}^2$, and even for $\mathbb{S}^2 \times \mathbb{S}^2$. In fact, their intersection forms are give by:
	\begin{itemize}
\item $Q_{\C\mathbb{P}^2 \# \overline{\C\mathbb{P}}^2} = \begin{bmatrix}
1 & 0\\
0 & -1
\end{bmatrix}$,
		\item $Q_{\mathbb{S}^2 \times \mathbb{S}^2} = \begin{bmatrix}
0 & 1\\
1 & 0
\end{bmatrix}$.
	\end{itemize}
\end{Example}


We will use the following classification theorems of smooth $4$--manifolds, due to Serre, Freedman and Donaldson:
\begin{Theorem}[Serre, Freedman, Donaldson]
Two smooth simply-connected $4$--manifolds are homeomorphic if and only if their intersection forms have the same rank, signature, and parity.
\end{Theorem}




\begin{Theorem}[Freedman's Classification Theorem \cite{Freedman}] \label{thm - Freedman Classification}
   For any integral symmetric unimodular form $Q$, there is a closed simply-connected topological $4$--manifold that has $Q$ as its intersection form.
\begin{itemize}
  \item If $Q$ is even, there is exactly one such manifold.
  \item If $Q$ is odd, there are exactly two such manifolds, at least one of which does not admit any smooth structures.
\end{itemize}  
\end{Theorem}

%

%

\begin{proof}[Proof of Theorem \ref{fib}]
	First, since $F$, $\mathfrak{T}_{\overline{\ell}}$ and $\mathfrak{F}$ are homotopically equivalent, their intersection forms are isomorphic. Since our description of $\mathfrak{T}_{\overline{\ell}}$ is more concrete, we will discuss $Q_{\mathfrak{T}_{\overline{\ell}}}$, and use that discussion to find the homeomorphism type of $\mathfrak{F}$.
	
	First, we know that the rank of $Q_{\mathfrak{T}_{\overline{\ell}}}$ is $2$, because $H^2(\mathfrak{T}_{\overline{\ell}}; \Z)\cong \Z \oplus \Z$.
	
	Second, since there is a self-homeomorphism $r$ of $\mathfrak{T}_{\overline{\ell}}$ reversing the orientation, and since the sign satisfies the following property: $\mathrm{sign}(Q_N) = -\mathrm{sign}(Q_{\overline{N}}),$ where $\overline{N}$ is $N$ with opposite orientation, we can see that the signature of the intersection form is $\mathrm{sign}(Q_{\mathfrak{T}_{\overline{\ell}}}) = 0$, hence $Q_{\mathfrak{T}_{\overline{\ell}}}$ is indefinite. 
	
	Let us show that $Q_{\mathfrak{T}_{\overline{\ell}}} = \begin{bmatrix}
1 & 0\\
0 & -1
\end{bmatrix}$ in an appropriate basis.
	

From the Mayer-Vietoris sequence in cohomology, we have
\begin{center}
\begin{tabular}{ccccccccc}
$H^1(Y;\Z)$\!\!\! & $\to$\!\! & $H^2(\mathfrak{T}_{\overline{\ell}};\Z)$\!\!\!                     & $\to$\!\! & $H^2(A;\Z) \oplus H^2(B;\Z)$\!\!\!       & $\to$\!\! & $H^2(Y;\Z)$\!\!\!       & $\to$\!\! & $H^3(\mathfrak{T}_{\overline{\ell}};\Z)$\\ 
$0$      & $\to$\!\! & $\mathbb{Z}\oplus\mathbb{Z}$ & $\to$\!\! & $\mathbb{Z}\oplus\mathbb{Z}$ & $\to$\!\! & $\mathbb{Z}_3$ & $\to$\!\! & $0$
\end{tabular}
\end{center}

Let $r$ be the orientation reversing involution of $\mathfrak{T}_{\overline{\ell}}$. We have that $r(A) = B$. The two maps $H^2(A;\Z) \to H^2(Y;\Z)$ and $H^2(B;\Z) \to H^2(Y;\Z)$ are either both zero or both non-zero. By the Mayer-Vietoris sequence, the map $H^2(A;\Z) \oplus H^2(B;\Z)\to H^2(Y;\Z)$ is onto, hence the two maps are both non-zero. We choose the generators of $H^2(A;\Z)$ and $H^2(B;\Z)$ such that both generators map to $1$ in $H^2(Y;\Z) = Z_3$. The map $\zeta\co H^2(A;\Z) \oplus H^2(B;\Z) \to H^2(Y;\Z)$ can then be written as
\[H^2(A;\Z) \oplus H^2(B;\Z) \ni (n,m) \longrightarrow \zeta(n,m) = n+m\ \ (\mathrm{mod}\ 3) \in H^2(Y;\Z) \,.\]
Hence the Mayer-Vietoris sequence and the injective map $\mu\co H^2(\mathfrak{T}_{\overline{\ell}}) \to H^2(A;\Z) \oplus H^2(B;\Z)$ identifies $H^2(\mathfrak{T}_{\overline{\ell}};\Z)$ with the subgroup
\[H^2(\mathfrak{T}_{\overline{\ell}};\Z) \cong \mathrm{Image}(\mu) = \mathrm{Ker}(\zeta) = \{\ (n,m)\in \mathbb{Z} \oplus \mathbb{Z} \ \mid\ n+m \equiv 0\ \ (\mathrm{mod}\ 3) \ \}\,. \]
A basis of $H^2(\mathfrak{T}_{\overline{\ell}};\Z)$ is given by the elements $v=(2,1)$ and $w=(1,2)$. We now express $Q_{\mathfrak{T}_{\overline{\ell}}}$ as a matrix in this basis:
\[ Q_{\mathfrak{T}_{\overline{\ell}}} = 
\begin{pmatrix}
x & z\\
z & y
\end{pmatrix}
\]
In order to compute $Q_{\mathfrak{T}_{\overline{\ell}}}$, we consider the elements $(3,0) = 2v-w$ and $(0,3) = 2w-v$. These two elements are $Q_{\mathfrak{T}_{\overline{\ell}}}$--orthogonal, because $A$ and $B$ retract to disjoint $2$--cycles in $H_2(\mathfrak{T}_{\overline{\ell}};\Z)$, and we have that $(3,0)$ maps to $0$ in $H^2(B;\Z)$ and $(0,3)$ maps to $0$ in $H^2(A;\Z)$. We have
\[Q_{\mathfrak{T}_{\overline{\ell}}}(2v-w,2w-v) = -2x - 2y + 5z = 0\]
\[Q_{\mathfrak{T}_{\overline{\ell}}}(2v-w,2v-w) =  4x +  y - 4z = q\]
\[Q_{\mathfrak{T}_{\overline{\ell}}}(2w-v,2w-v) =   x + 4y - 4z = -q,\]
where $q \in \mathbb{Z}$ is the norm of $(3,0)$. The norm of $(0,3)$ is then $-q$, because $r(3,0) = (0,\pm 3)$, and $r$ reverses the orientation. 

The determinant of the $3 \times 3$ matrix of the coefficients is $27 \neq 0$, hence the system of equations has at most one solution. An explicit solution is given by $x = \frac{q}{3}, y = -\frac{q}{3}, z=0$. Since $Q$ is unimodular, this implies $q = \pm 3$ and we conclude that
\[ Q_{\mathfrak{T}_{\overline{\ell}}} = 
\begin{pmatrix}
1 & 0\\
0 & -1
\end{pmatrix}~.
\]

Since $\mathfrak{T}_{\overline{\ell}}$ and $\mathfrak F$ are homotopy equivalent, we have proven that $\mathfrak F$ has the same intersection matrix as $\C\mathbb{P}^2 \# \overline{\C\mathbb{P}}^2$. Since $\mathfrak F$ is smooth, Freedman's classification theorem Theorem tells us that $\mathfrak F$ is homeomorphic to $\C\mathbb{P}^2 \# \overline{\C\mathbb{P}}^2$ (hence that $\mathfrak{T}_{\overline{\ell}}$ and $F$ are homotopy equivalent to $\C\mathbb{P}^2 \# \overline{\C\mathbb{P}}^2$).
\end{proof}

\bibliographystyle{amsalpha}
\bibliography{Biblio} 

\end{document}